\documentclass[english]{amsart}
\usepackage{amsmath} 
\usepackage{amssymb} 
\usepackage{amsthm}
\numberwithin{equation}{section}
\usepackage{graphicx}
\usepackage{geometry}
\usepackage{xcolor}
\usepackage{esint}
\usepackage{bm}
\usepackage{tikz}
\usepackage{comment}
\usepackage[nocompress]{cite}
\usepackage[colorlinks=true, allcolors=blue]{hyperref}
\usepackage{cleveref}
\newcommand{\OmegaT}{\Omega_T}
\newcommand{\uvec}{\mathbf{u}}

\newcommand{\vvec}{\mathbf{v}}

\newcommand{\varphivec}{\bm{\varphi}}

\newcommand{\RR}{\mathbb{R}}
\newcommand{\NN}{\mathbb{N}}
\renewcommand{\div}{\mathrm{div}_x}
\newcommand{\nablax}{\nabla_x}
\newcommand{\Svisc}{\mathbb{S}}
\newcommand{\Deltax}{\Delta_x}
\newcommand{\dd}{\mathrm{d}}
\newcommand{\weak}{\rightharpoonup}
\newcommand{\weakstar}{\overset{\ast}{\weak}}
\newcommand{\tgamma}{{\tilde{\gamma}}}
\newcommand{\dt}{\mathrm{d}t}
\newcommand{\dx}{\mathrm{d}x}
\newcommand{\Cw}{C_{\mathrm{w}}} %weakly continuous functions
\newcommand{\shearvisc}{\mu} %shear viscosity coefficient
\newcommand{\bulkvisc}{\eta} %bulk viscosity coefficient
\newcommand{\coup}{\varsigma} %elliptic relaxation coefficient
\newcommand{\paracoup}{\beta} %parabolic relaxation coefficient
\newcommand{\Peff}{p_\coup} %artificial pressure function
\newcommand{\eps}{\varepsilon} %hole size of porous domain
\newcommand{\rhoeps}{\rho_{\eps}} %density on porous domain
\newcommand{\uveceps}{\uvec_{\eps}} %velocity on porous domain
\newcommand{\ceps}{c_{\eps}} %relaxation parameter on porous domain
\newcommand{\trhoeps}{\tilde{\rho}_{\eps}} %zero extension of density
\newcommand{\tuveceps}{\tilde{\uvec}_{\eps}} %zero extension of velocity
\newcommand{\tceps}{\tilde{c}_{\eps}} %zero extension of relaxation parameter
\newcommand{\Eeps}{\mathbb{E}_{\eps}} %extension operator for relaxation parameter
\newcommand{\Omegaeps}{\Omega_{\eps}} %porous domain
\newcommand{\Bog}{\mathcal{B}} %Bogovskii operator on unperforated domain
\newcommand{\Bogeps}{\Bog_{\eps}} %Bogovskii operator on porous domain
\newcommand{\dimension}{d} %dimension variable
\newcommand{\Bogepsi}{\mathcal{B}_{\eps,i}}
\newcommand{\Aepsi}{A_i^\eps}
\newcommand{\Xieps}{\Xi^\eps}
\newcommand{\Xiepsi}{\Xi_i^\eps}
\newcommand{\thetaeps}{\vartheta^\eps}
\newcommand{\thetaepsi}{\vartheta^\eps_i}
\newcommand{\xepsi}{x_i^\eps}
\newcommand{\Bepsi}{B_i^\eps}
\newcommand{\Lepsi}{L_{i,\eps}}
\newcommand{\gvec}{\mathbf{g}}

\newtheorem{theorem}{Theorem}[section]
\newtheorem*{theorem*}{Theorem}

\newtheorem{lemma}{Lemma}[section]
\theoremstyle{remark}
\newtheorem{remark}{Remark}[section]
\allowdisplaybreaks
\theoremstyle{definition}
\newtheorem{definition}{Definition}[section]

\title[Homogenization of relaxed compressible viscous two-phase fluid model]{Homogenization of a relaxed compressible viscous two-phase fluid model in a domain with very tiny holes}

\author{Florian Oschmann}
\address{Faculty of Mathematics and Physics of the Charles University, Sokolovsk\'a 49/83, CZ-186 00 Praha, Czech Republic}
\email{\href{mailto:florian.oschmann@matfyz.cuni.cz}{florian.oschmann@matfyz.cuni.cz}}
\author{Florian Wendt}
\address{Institute of Applied Analysis and Numerical Simulation, University of Stuttgart, Pfaffenwaldring~57, D-70569 Stuttgart, Germany}
\email{\href{mailto:florian.wendt@mathematik.uni-stuttgart.de}{florian.wendt@mathematik.uni-stuttgart.de}}
\date{\today} 
\begin{document}
\begin{abstract}
    We study an approximate system for the compressible Navier--Stokes--Korteweg equations in a bounded domain periodically perforated by small obstacles.
    As the size of the obstacles decrease much faster to zero than their mutual distances, we show in spatial dimension two and three that the limiting system remains unchanged.
    Our result applies for a large class of monotone and non-monotone pressure functions.
    In particular, it holds in the physically relevant case when the approximate system is used to describe the dynamics of a compressible viscous two-phase fluid extending the corresponding homogenization results known for the compressible Navier--Stokes equations in a single-phase setting.
\end{abstract}

\maketitle

%\tableofcontents

\section{Introduction}\label{sec:Intro}
Accurately describing the effective behavior of small scale interactions in a macroscopic fluid is of great interest for a wide range of applications.
Examples include the description of aerosols and sprays, but also gas flow through rock.
The derivation of such effective systems is usually termed \emph{homogenization} and mainly depends on the relation between the size of the obstacles and their mutual distances as well as on the fluid flow model under consideration.
In a single-phase setting, there is nowadays a vast literature on rigorous analytical results for the homogenization of Stokes, Navier--Stokes, and Navier--Stokes--Fourier equations, both incompressible and compressible, stationary and time-dependent.
All these results can roughly be sorted by the holes' size: if $\eps>0$ denotes the holes' mutual distance and $\eps^\alpha$ their radius, we can state in 3D:
\begin{itemize}
    \item if $1 \leq \alpha < 3$, then the limiting system is governed by Darcy's law;
    \item if $\alpha = 3$, then the limiting system is governed by Brinkman's law;
    \item if $\alpha > 3$, then the limiting system remains unchanged.
\end{itemize}
Analogues for the 2D setting are available as well. Additionally, there are various results focusing also on different parameters such as vanishing viscosity, low Mach number limit, or random distribution of holes.
Without being exhaustive, we just cite a few and refer the interested reader to \cite{Allaire1990a, Allaire1990b, Allaire1989, CioranescuDonatoEne1996, FeireislNamlyeyevaNecasova2016, GBNNR2026, Lu2020, Lu2021, Hoefer2022, HoeferLuOschmann2026, LuQian2024, LuYang2023, Mikelic1991,Shen2022,Marusic-PalokaMikelic1996} for incompressible fluids, and to \cite{BasaricChaudhuri2024, BellaOschmann2023, BLMO2025, DieningFeireislLu2017, FeireislLu2015, FeireislNovotnyTakahashi2010, HoeferKowalczykSchwarzacher2021, LuPokorny2021, LuSchwarzacher2018, Masmoudi2002, NecasovaOschmann2023, NecasovaPan2022, Oschmann2022, PokornySkrisovsky2021a} for compressible fluids, as well as to the references contained therein.
\par
Far less is known about the homogenization of fluids that can occur in two different phases.
Up to the authors' knowledge, there is only one rigorous homogenization result \cite{RohdeWolff2020} for a modification of the 3D compressible Navier--Stokes--Korteweg equations (NSKE) with a Van-der-Waals type pressure function.
In this reference, the authors focus on large holes with $\alpha=1$ and the modified capillary force given by some smooth convolution operator.
The advantage of this operator is that all possibly harming derivatives do not fall on the density but on a smooth interaction kernel inside the convolution operator. 
Their outcome is a compressible Darcy-type law inheriting the interaction kernel.
\par
In this paper we extend the landscape of rigorous homogenization results for compressible viscous two-phase fluid models.
As a two-phase fluid model on the pore-scale, we choose a parabolic relaxation system of the NSKE recently proposed in \cite{HKMR2020} with a Van-der-Waals type pressure function.
This system depends on two relaxation coefficients and serves as an approximate model to the compressible NSKE if the coefficients are chosen appropriately (see \Cref{sec:modelWkSolMain} for more details).
Compared to the result in \cite{RohdeWolff2020}, we are rather focusing on two different aspects: first, we consider tiny holes with $\alpha>d$, where $d \in \{2,3\}$ denotes the spatial dimension; second, the capillary force is replaced by some relaxation parameter which satisfies a linear parabolic equation with Neumann boundary condition on the holes.
This allows for weaker assumptions on smoothness for the capillary force at the cost of an additional boundary condition. Moreover, to the best of the authors' knowledge, in the regime of tiny holes there is no homogenization result for pressures that depend on the density in a \emph{non-monotone} fashion, which in view of applications and such-called Van-der-Waals type pressures is rather unsatisfactory; in here, we consider homogenization of such non-monotone pressure functions for the first time.
\par

\subsection*{Notations}
For a bounded domain $U \subseteq \RR^d$ with $d \in \{2,3\}$ we use the standard notation for Lebesgue and Sobolev spaces on $U$.
The space of smooth compactly supported functions is denoted by $\mathcal{D}(U)$, and the corresponding space of distributions on $U$ is denoted by $\mathcal{D}^\prime(U)$.
For an integrable function $f \in L^1(U)$, we denote its zero extension onto $\RR^d$ by $\tilde{f}$, i.e.,
\begin{align}\label{zero ext}
    \tilde{f}=f \quad \text{in } U,
    &&
    \tilde{f} = 0 \quad \text{in } \RR^d\setminus U.
\end{align}
Moreover, if $|U|>0$, we introduce the mean-value over $U$ as
\begin{equation*}
    \fint_U f := \frac{1}{|U|} \int_U f(x) \, \dd x.
\end{equation*}
Here, $|U|$ denotes the Lebesgue-measure of the domain $U$.
For some positive time $T>0$, we use the notation $U_T:=(0,T)\times U$.
The space of functions on $[0,T]$ ranging into some Banach space $X$ continuously with respect to the weak topology on $X$ will be denoted by $\Cw([0,T];X)$. For $s \in [1,\infty]$, we sometimes write instead of $L^s(U;\RR^d)$ or $L^s(U;\RR^{d\times d})$ just $L^s(U)$, if no ambiguities appear.
The same convention applies for other standard function spaces on $U$.
We further use the notation
\begin{equation*}
    s^\prime := \begin{cases}
        \frac{s}{s-1} &\text{if } s \in (1,\infty),
        \\
        \infty &\text{if } s=1,
        \\
        1 &\text{if } s=\infty,
    \end{cases}
    \qquad 
    s^\ast := \begin{cases}
        \frac{\dimension s}{\dimension -s} &\text{if } s \in [1,\dimension),
        \\
        \infty &\text{if } s \in [d,\infty]
    \end{cases}
\end{equation*}
Finally, if $s >1$, we denote by $s^-$ any number lying in the interval $[1,s)$.

\subsection*{Organization of the Paper}
In \Cref{sec:modelWkSolMain} we first introduce the perforated domain and state an extension lemma for functions defined on this domain.
Then we introduce the parabolic relaxation system of the NSKE, specify our assumptions on the pressure functions, and introduce the concept of finite energy weak solutions corresponding to the relaxation system.
At the end of \Cref{sec:modelWkSolMain} we state our main homogenization result.
In \Cref{sec:Hom3D} and \Cref{sec:Hom2D} we prove our main result in dimension three and two, respectively.
We close this work with some conclusions and directions for future works in \Cref{sec:Concl}.

\section{The two-phase model, weak solutions, and the main result}\label{sec:modelWkSolMain}
In this section, we specify our assumptions on the perforated domain, introduce the two-phase fluid model, and state our main result.
We begin with the perforated domain.
\subsection{The perforated domain}\label{subsec:perforated domain}
We focus on a bounded domain $\Omega \subset \RR^\dimension$, $\dimension\in \{2,3\}$, with smooth boundary, periodically\footnote{The assumption of periodicity can easily be relaxed, see e.g. \cite[Section~2.1]{OschmannPokorny2023}.} perforated by tiny holes of size
\begin{equation*}
    a_\eps:=
    \begin{cases}
        \eps^{\alpha},\,\,\alpha>3 &\text{if } \dimension=3,
        \\
        \exp(-\eps^{-\alpha}),\,\, \alpha>2 &\text{if } \dimension=2,
    \end{cases}
\end{equation*}
where $\eps>0$ denotes the mutual distance between the holes.
More precisely, we introduce for $\eps \in (0,1)$ the \emph{perforated domain} $\Omegaeps$ as
\begin{align}\label{def:domain}
    \Omega_\eps = \Omega \setminus \bigcup_{i \in K_\eps} B_{a_\eps}(\eps x_i), \qquad x_i \in \mathbb{Z}^\dimension, \qquad K_\eps = \{i \in \mathbb{Z}: B_{\eps}(\eps x_i) \Subset \Omega \},
\end{align}
see Figure~\ref{fig:perfDom}.
\begin{figure}[ht]
\centering
\begin{tikzpicture}[scale=.7]
\draw[black, thick] plot [smooth cycle] coordinates {(0,5) (-1,4) (0,-1) (8,-0.3) (7.7,2) (7,5) (3,5)};
\foreach \a in {0, 1, ..., 6}
	\foreach \b in {0, 1, 2, 3, 4}
        \draw (\a,\b) circle (.3cm);
\foreach \b in {0, 1, 2, 3}
\draw (7,\b)
        circle (.3cm);
\draw (1,4) -- (2,4);
\node at (1.5,4) [anchor=south] {$\eps$};
\draw (0,3) -- (-1.5,3);
\node at (-1.5,3) [anchor=east] {$B_{a_\eps}(\eps x_i),$};
\node at (-1.5,2) [anchor=east] {$x_i\in \mathbb{Z}^\dimension$};
\node at (-2,5) {$\Omega_\eps \subset \RR^\dimension$};
\end{tikzpicture} \caption{The perforated domain $\Omega_\eps$.} \label{fig:perfDom}
\end{figure}

It is clear that the number and measure of holes satisfy
\begin{align}\label{porous dom number density}
    |K_\eps| \leq C \eps^{-\dimension}, && \left| \bigcup_{i \in K_\eps} B_{a_\eps}(\eps x_i) \right| \leq C (a_\eps / \eps)^d \to 0,
\end{align}
for some constant $C>0$ that does not depend on $\eps$.
To state our main result at the end of this section, we need the following extension result from\cite[Lemma~2.1]{PokornySkrisovsky2021a} and \cite[Lemma~3.1]{LuPokorny2021}.
\begin{lemma}\label{lem:extEeps}
Let $\dimension  \in \{2,3\}$ and let $\Omega \subseteq \RR^{\dimension}$ be a bounded domain with smooth boundary.
For $\eps>0$, let $\Omega_\eps$ be defined as in \eqref{def:domain}. Then there exists an extension operator $\Eeps:W^{1,2}(\Omega_\eps)\to W^{1,2}(\Omega)$ such that, for any $\phi\in W^{1,2}(\Omega_\eps)$,
\begin{align*}
&\Eeps\phi=\phi\text{ in } \Omega_\eps,
\quad 
\|\nablax \Eeps\phi\|_{L^2(\Omega)}\leq C\,\|\nablax\phi\|_{L^2(\Omega_\eps)},
\end{align*}
and moreover, for any $1\leq s\leq\infty$,
\begin{align*}
\| \Eeps\phi\|_{L^s(\Omega)}\leq C\,\|\phi\|_{L^s(\Omega_\eps)},
\end{align*}
for some constant $C>0$ that does not depend on $\eps$.
%Furthermore, there exists an operator ${\tilde{E}_\eps:H^1_{\geq 0}(D_\eps)\to H^1_{\geq 0}(D)}$ with the same properties as above. Here $H^1_{\geq 0}$ denotes the Sobolev space of all non-negative functions in $H^1$. In particular, one may choose $\tilde{E}_\eps\phi:=\max\{0,E_\eps\phi\}$.
\end{lemma}
For $f \in L^1(0,T;W^{1,2}(\Omega_\eps))$, we will use the symbol $\Eeps$ to denote the extension that acts on the spatial variable only, that is,
\begin{equation*}
    (\Eeps f)(t,x) := \Eeps(f(t,\cdot))(x) \quad \text{for a.e.~} (t,x) \in \OmegaT.
\end{equation*}
In the same spirit, we will use the symbol $\tilde{\cdot}$ to denote the spatial zero extension for a function $g\in L^1(0,T;L^1(\Omegaeps))$, that is,
\begin{align*}
    \tilde{g}(t,x):=\widetilde{g(t,\cdot)}(x) \quad \text{for a.e.~} (t,x) \in \OmegaT,
\end{align*}
with the spatial zero extension $\widetilde{g(t,\cdot)}$ being defined in \eqref{zero ext}.
\subsection{The two-phase fluid model}
For some positive time $T>0$, we consider a parabolic relaxation system of the NSKE recently proposed in \cite{HKMR2020}.
This system governs the dynamics of a two-phase fluid occupying the perforated domain $\Omegaeps$ by the \emph{density} $\rhoeps\colon [0,T]\times\Omegaeps \to \RR_{\geq0}$, the \emph{velocity} $\uveceps \colon [0,T]\times\Omegaeps \to \RR^\dimension$, and the \emph{relaxation parameter} $\ceps \colon [0,T]\times\Omegaeps \to \RR$ which satisfy
\begin{equation}\label{rNSKE}
    \left\{
\begin{aligned}
    &\partial_t \rhoeps 
    +
    \div(\rhoeps\uveceps)
    = 0
    &&\text{in } (0,T)\times\Omegaeps,
    \\[0.15cm]
    &\partial_t(\rhoeps\uveceps)
    +
    \div(\rhoeps\uveceps\otimes\uveceps)
    +
    \nablax p(\rhoeps)
    \\
    &\quad=
    \div\Svisc(\nablax\uveceps)
    +
    \coup \rhoeps\nablax(\ceps-\rhoeps)
    &&\text{in } (0,T)\times\Omegaeps,
    \\[0.15cm]
    &\paracoup\partial_t \ceps 
    -
    \kappa \Deltax \ceps
    +
    \coup(\ceps-\rhoeps) 
    = 
    0
    &&\text{in } (0,T)\times\Omegaeps,
    \\[0.15cm]
    &{\uveceps}=0, \qquad \nablax \ceps \cdot \mathbf{n}_{\partial\Omegaeps} = 0
    &&\text{on } (0,T)\times\partial\Omegaeps,
    \\[0.15cm]
    &\rhoeps(0)=\rho_{0\eps},\qquad (\rhoeps\uveceps)(0)=(\rho\uvec)_{0\eps},\qquad \ceps(0)=c_{0 \eps}
    &&\text{in } \Omegaeps.
\end{aligned}
\right.
\end{equation}
In \eqref{rNSKE} we have used the notation
\begin{equation*}
    \Svisc(M) := 2\shearvisc\Big(\frac{M + M^\mathrm{T}}{2} - \frac{1}{\dimension} \mathrm{Tr}(M) \Big) 
    +
    \bulkvisc \,\mathrm{Tr}(M) 
    \qquad \forall\, M \in \RR^{d\times d},
\end{equation*}
with constants $\shearvisc>0$ and $\bulkvisc \geq 0$ denoting the shear and bulk viscosity of the two-phase fluid, respectively.
Moreover, $\kappa >0$ denotes the constant \emph{capillarity coefficient} and $\coup,\paracoup>0$ denote the constant \emph{coupling coefficients}.
For the pressure function $p\colon[0,\infty)\to[0,\infty)$ we assume the decomposition
\begin{equation}\label{dec prs reg}
    p= h+q,
    \qquad 
    h\in C^0([0,\infty))\cap C^1((0,\infty)),
    \quad 
    q\in C^\infty_c([0,\infty)),
\end{equation}
with $h$ and $q$ satisfying
\begin{equation}\label{dec prs grwth}
    \begin{aligned}
        h(0)=0=q(0),
        \quad 
        q\leq 0,
        \quad 
        h(r) \leq C_h + d_h r^{\gamma},
        \quad 
        c_hr^{\gamma-1}
        \leq 
        h^\prime(r)
        \quad \forall\, r \in (0,\infty),
    \end{aligned}
\end{equation}
for some positive constants $\gamma \in (1,\infty)$ and $C_h,c_h,d_h\in (0,\infty)$.
By setting $q=0$ in \eqref{dec prs reg}--\eqref{dec prs grwth} we account for a large class of monotone pressure functions that are frequently used in the context of the single-phase compressible Navier--Stokes equations (see e.g.~\cite{FeireislNovotnyPetzeltova2001,HoeferNecasovaOschmann2026,Masmoudi2002,OschmannPokorny2023}).
In a two-phase setting, the pressure function $p\colon [0,\infty) \to [0,\infty)$ is assumed to be of \emph{Van-der-Waals type}, that is, smooth, monotonically decreasing on some compact interval lying in $(0,\infty)$, and monotonically increasing on its complement.
Note that \eqref{dec prs reg}--\eqref{dec prs grwth} allows $p$ to have a compact non-monotone region and thus, to be of Van-der-Waals type.
For more details concerning the two-phase modeling we refer to \cite{HKMR2020}.
\par
From now on, we call system \eqref{rNSKE} the \emph{relaxed Navier--Stokes--Korteweg equations (rNSKE)}.
Formally, one can show that solutions of the rNSKE converge in the relaxation limit $\coup \to \infty$ and $\paracoup\to 0$ to a solution of the NSKE
\begin{equation*}
    \left\{
\begin{aligned}
    &\partial_t \rhoeps 
    +
    \div(\rhoeps\uveceps)
    = 0
    &&\text{in } (0,T)\times\Omegaeps,
    \\[0.15cm]
    &\partial_t(\rhoeps\uveceps)
    +
    \div(\rhoeps\uveceps\otimes\uveceps)
    +
    \nablax p(\rhoeps)
    =
    \div\Svisc(\nablax\uveceps)
    +
    \kappa \rhoeps \nablax \Deltax \rhoeps
    &&\text{in } (0,T)\times\Omegaeps,
    \\[0.15cm]
    &{\uveceps}=0, \qquad \nablax \rhoeps \cdot \mathbf{n}_{\partial\Omegaeps} = 0
    &&\text{on } (0,T)\times\partial\Omegaeps,
    \\[0.15cm]
    &\rhoeps(0)=\rho_{0\eps},\qquad (\rhoeps\uveceps)(0)=(\rho\uvec)_{0\eps},
    &&\text{in } \Omegaeps.
\end{aligned}
\right.
\end{equation*}
In this relaxation limit, the corresponding sequence of relaxation coefficients converges to the density of the NSKE. This behavior has also been observed numerically in \cite{HKMR2020}.
Recently, a rigorous convergence result for the relaxation limit in the framework of finite energy weak solutions has been obtained in \cite{ChaudhuriRohdeWendt2025} based on a relative energy approach.
For more details concerning the NSKE we refer to \cite{AndersonMcFaddenWheeler1998,DunnSerrin1985}.
\par
Introducing the \emph{artificial pressure function}
\begin{equation}\label{p artificial}
    \Peff(r):=p(r)+\frac{\coup}{2} r^2 \qquad \forall\, r \in [0,\infty)
\end{equation}
we can rewrite the momentum equation of the rNSKE \eqref{rNSKE} as
\begin{equation*}
    \partial_t(\rhoeps\uveceps) + \div(\rhoeps\uveceps\otimes\uveceps) +\nablax p_\coup(\rhoeps) 
    =
    \div\Svisc(\nablax\uveceps) + \coup\rhoeps\nablax \ceps.
\end{equation*}
In a two-phase setting, the artificial pressure function $\Peff$ is monotonically increasing, provided $\coup>0$ is large enough.
This renders the underlying first-order system hyperbolic and facilitates the numerical treatment of the rNSKE in comparison with the NSKE.
We refer to \cite{HKMR2020} for more details including numerical simulations of a compressible viscous two-phase fluid using this reformulation.
\par
The relaxation parameter is a purely artificial quantity.
It is also possible to use an elliptic relaxation equation for the relaxation parameter, i.e., \eqref{rNSKE} with $\paracoup=0$ and no initial condition for the relaxation parameter $c$.
Such a relaxation system has been proposed in \cite{Rohde2010} and was numerically investigated in \cite{NeusserRohdeSchleper2015}.
For the corresponding relaxation limit a rigorous convergence result has been obtained in \cite{GiesselmannLattanzioTzavaras2017}.

\subsection{Finite energy weak solutions and the main result}
For our homogenization result we use the framework of finite energy weak solutions to the rNSKE.
This concept is well-known in the context of the compressible Navier--Stokes equations (see~\cite{FeireislNovotnyPetzeltova2001,Lions1998}).
The corresponding definition for the rNSKE has been motivated in \cite{OschmannWendt2026-exist} and reads on the perforated domain $\Omegaeps$ as follows:
\begin{definition}
    Let $d\in \{2,3\}$, $\Omegaeps$ be defined as in \eqref{def:domain}, and $T,\coup,\paracoup,\kappa,\shearvisc>0$, $\bulkvisc\geq 0$.
    Assume that $p$ satisfies \eqref{dec prs reg}--\eqref{dec prs grwth} with $\gamma \in (1,\infty)$ and let initial conditions $\rho_{0\eps}\in L^{\tgamma}(\Omegaeps)$, $(\rho\uvec)_{0\eps}\in L^{\frac{2\tgamma}{\tgamma+1}}(\Omegaeps)$, $c_{0\eps}\in W^{1,2}(\Omegaeps)$, where $\tgamma:=\max\{2,\gamma\}$, be given satisfying
    \begin{equation}\label{wkSol IC compat}
        \rho_{0\eps}\geq 0 \quad \text{a.e.},
        \qquad 
        (\rho\uvec)_{0\eps} = 0
        \quad \text{on } \{\rho_{0\eps}=0\},
        \qquad 
        \frac{\big|(\rho\uvec)_{0\eps}\big|^2}{\rho_{0\eps}}\in L^1(\Omegaeps).
    \end{equation}
    Then we call the triplet $(\rhoeps,\uveceps,\ceps)$ a \emph{finite energy weak solution} to the rNSKE on $(0,T)\times\Omegaeps$ emanating from $(\rho_{0\eps},(\rho\uvec)_{0\eps},c_{0\eps})$ if:
    \begin{itemize}
        \item The solution satisfies the regularity
        \begin{equation}\label{wkSol reg}
            \begin{aligned}
                &\rhoeps \in \Cw([0,T];L^{\tgamma}(\Omegaeps)),
                && 
                \rhoeps \geq 0 \,\,\, \text{a.e. in~} (0,T)\times\Omegaeps,
                \\
                &\uveceps \in L^2(0,T;W^{1,2}_0(\Omegaeps)),
                &&
                (\rhoeps\uveceps)\in \Cw([0,T];L^{\frac{2\tgamma}{\tgamma+1}}(\Omegaeps)),
                \\
                &\partial_t c_\eps \in L^2(\OmegaT),
                &&
                c_\eps \in C([0,T];W^{1,2}(\Omegaeps))\cap L^2(0,T;W^{2,2}(\Omegaeps));
            \end{aligned}
        \end{equation}
        \item the solution satisfies $\eqref{rNSKE}_1$--$\eqref{rNSKE}_3$ in the sense of integral identities, that is,
        \begin{align}
            &\int_0^T \int_{\Omegaeps}
            \rhoeps \partial_t \varphi + \rhoeps\uveceps \cdot \nablax \varphi 
            \, \dd x \, \dd t
            = 0,\label{wkSol cont}
            \\
            &\int_0^T \int_{\Omegaeps}
            \paracoup \ceps \partial_t \varphi - \kappa \nablax \ceps \cdot \nablax \varphi - \coup (\ceps-\rhoeps) \varphi 
            \, \dd x \, \dd t=0\label{wkSol parab}
        \end{align}
        for any test function $\varphi \in \mathcal{D}(\RR_T^\dimension)$, and
        \begin{equation}\label{wkSol mom}
            \begin{aligned}
                &\int_0^T \int_{\Omegaeps}
                \rhoeps \uveceps \cdot \varphivec + \rhoeps\uveceps \otimes \uveceps : \nablax \varphivec + \Peff(\rhoeps) \div\varphivec 
                \, \dd x \, \dd t
                \\
                &\quad =
                \int_0^T \int_{\Omegaeps} 
                \Svisc(\nablax\uveceps):\nablax \varphivec
                -
                \coup\rhoeps \nablax \ceps \cdot \varphivec
                \, \dd x \, \dd t
            \end{aligned}
        \end{equation}
        for any test function $\varphivec \in \mathcal{D}((0,T)\times\Omegaeps;\RR^\dimension)$, where $p_\coup$ is defined in \eqref{p artificial};
        \item the solution satisfies the initial conditions
        \begin{equation}\label{wkSol IC}
            \rhoeps(0)=\rho_{0\eps},
            \qquad 
            (\rhoeps\uveceps)(0)=(\rho\uvec)_{0\eps},
            \qquad 
            \ceps(0)=c_{0\eps}
            \qquad 
            \text{a.e.~in } \Omegaeps;
        \end{equation}
        \item the energy inequality
        \begin{equation}\label{wkSol energy}
        \begin{aligned}
            &-\int_0^T \partial_t \psi\int_{\Omegaeps} \frac{1}{2}\rhoeps|\uveceps|^2 + W(\rhoeps) + \frac{\coup}{2}|\rhoeps-\ceps|^2 + \frac{\kappa}{2}|\nablax \ceps|^2 
            \, \dd x \, \dd t
            \\
            &\quad 
            +\int_0^T \psi \int_{\Omegaeps}\Svisc(\nablax\uveceps):\nablax\uveceps + \paracoup |\partial_t \ceps|^2 
            \, \dd x \, \dd t
            \\
            &\leq 
            \psi(0) \int_{\Omegaeps} \frac{\big|(\rho\uvec)_{0\eps}\big|^2}{2\rho_{0\eps}} + W(\rho_{0\eps}) + \frac{\coup}{2}|\rho_{0\eps} - c_{0\eps}|^2 + \frac{\kappa}{2}|\nablax c_{0\eps}|^2 
            \, \dd x
        \end{aligned}
        \end{equation}
        holds for any $\psi \in C^\infty_c([0,T))$ with $\psi \geq 0$, where
        \begin{equation*}
            W\colon [0,\infty)\to \RR,
            \qquad 
            r\mapsto W(r):=r\int_1^{r}\frac{p(z)}{z^2}\, \dd z
        \end{equation*}
        denotes the pressure potential corresponding to $p$.
    \end{itemize}
\end{definition}
Concerning the global-in-time existence of finite energy weak solutions to the rNSKE we have the following result:
\begin{theorem}[see {\cite[Theorem~2.1]{OschmannWendt2026-exist}}]
    Let $d\in \{2,3\}$, $\Omegaeps$ be defined as in \eqref{def:domain}, and $T,\coup,\paracoup,\kappa,\shearvisc>0$, $\bulkvisc\geq 0$.
    Assume that $p$ satisfies \eqref{dec prs reg}--\eqref{dec prs grwth} with $\gamma \in (1,\infty)$ and let initial conditions $\rho_{0\eps}\in L^{\tgamma}(\Omegaeps)$, $(\rho\uvec)_{0\eps}\in L^{\frac{2\tgamma}{\tgamma+1}}(\Omegaeps)$, $c_{0\eps}\in W^{1,2}(\Omegaeps)$, $\tgamma:=\max\{2,\gamma\}$, be given satisfying \eqref{wkSol IC compat}.
    Then there exists a finite energy weak solution $(\rhoeps,\uveceps,\ceps)$ to the rNSKE on $(0,T)\times\Omegaeps$ emanating from $(\rho_{0\eps},(\rho\uvec)_{0\eps},c_{0\eps})$.
\end{theorem}
We will analyze the behavior of the rNSKE as $\eps \to 0$ in the framework of finite energy weak solutions. Similarly to homogenization results in the literature (see e.g. \cite{BellaOschmann2023, DieningFeireislLu2017, OschmannPokorny2023} for $d=3$ and \cite{NecasovaOschmann2023,NecasovaPan2022,Bravin2024} for $d=2$), since the holes are tiny (that is, $\alpha > 3$ for $d=3$ and $\alpha>2$ for $d=2$), we don't expect that they will influence the limiting equations.
In fact, we will show the following result which is the main result of this paper.
\begin{theorem}\label{thm:Hom}
    Let $\coup,\paracoup,\kappa,T,\mu>0$, $\eta \geq 0$, and let $\Omega \subseteq \RR^\dimension$, $\dimension\in \{2,3\}$, be a bounded domain with smooth boundary.
    For $\eps \in (0,1)$, let $\Omegaeps$ be defined as in \eqref{def:domain}.
    Assume that $p$ satisfies \eqref{dec prs reg}--\eqref{dec prs grwth}, and suppose that $\gamma$ and $\alpha$ satisfy the relations
    \begin{align}
        &\gamma > 3, \qquad  \alpha > \max\Big\{3, \frac{2\gamma-3}{\gamma-3}\Big\}
        && \text{if } \dimension = 3,\label{thm:Hom3D restr gamma,alpha}
        \\
        &\gamma>2, \qquad \alpha>2
        &&\text{if } \dimension =2.\label{thm:Hom2D restr gamma,alpha}
    \end{align}
    For $\eps\in(0,1)$, let initial conditions $(\rho_{0\eps},(\rho\uvec)_{0\eps},c_{0\eps})\in L^\gamma(\Omega)\times L^{\frac{2\gamma}{\gamma+1}}(\Omega;\RR^\dimension)\times W^{1,2}(\Omega)$ be given satisfying \eqref{wkSol IC compat} and
    \begin{equation}\label{thm:Hom IC}
    \begin{aligned}
        &\tilde{\rho}_{0\eps} \to \rho_0 \quad \text{in } L^\gamma(\Omega),
        \quad 
        \frac{|\widetilde{(\rho\uvec)}_{0\eps}|^2}{\tilde{\rho}_{0\eps}}\to \frac{|(\rho\uvec)_0|^2}{\rho_0}
        \quad \text{in } L^1(\Omega),
        \\
        &\Eeps c_{0\eps} \to c_0 \quad \text{in } W^{1,2} (\Omega),
    \end{aligned}
    \end{equation}
    where $\tilde{\cdot}$ denotes the spatial zero extension, and $\Eeps$ denotes the extension operator from \Cref{lem:extEeps}.
    Moreover, let $(\rhoeps,\uveceps,\ceps)$ be a finite energy weak solution to the rNSKE on $(0,T)\times\Omegaeps$ emanating from $(\rho_{0\eps},(\rho\uvec)_{0\eps},c_{0\eps})$.    
    Then there exists a non-relabeled subsequence such that
    \begin{align*}
        &\tilde{\rho}_\eps \to \rho \quad \text{in } \Cw([0,T]; L^\gamma(\Omega)),
        &&
        \tilde{\uvec}_\eps \weak \uvec \quad  \text{in } L^2(0,T; W^{1,2}_0(\Omega)),\\
        &\Eeps c_\eps \weakstar c \quad \text{in } L^\infty(0,T; W^{1,2}(\Omega)),
        &&
        (\Eeps\ceps,\tceps) \to (c,c) \quad \text{in } C([0,T];L^2(\Omega;\RR^2)),
    \end{align*}
    and the triplet $(\rho, \uvec, c)$ is a finite energy weak solution of the rNSKE on $(0,T)\times\Omega$ emanating from $(\rho_0,(\rho\uvec)_0,c_0)$.
\end{theorem}

\section{Homogenization in 3D}\label{sec:Hom3D}
To ease the notation, we denote in the sequel by $C>0$ a positive constant that may be different from line to line, but does not depend on $\eps$.
According to the different settings depending on the dimension $\dimension\in \{2,3\}$, we will split the proof of Theorem~\ref{thm:Hom} into two parts.
This section is devoted to the proof of the case $d=3$.
\subsection{Uniform estimates}\label{subsec:unifBds}
Let $(\rhoeps,\uveceps,\ceps)$ be a finite energy weak solution emanating from the initial data $(\rho_{0\eps},(\rho\uvec)_{0\eps},c_{0\eps})$ that satisfies the convergences in \eqref{thm:Hom IC}.
In order to perform the homogenization limit $\eps \to 0$, we have to derive appropriate bounds for the solution $(\rhoeps,\uveceps,\ceps)$ that are uniform with respect to the homogenization parameter $\eps>0$.
The first set of uniform bounds follows immediately from the energy inequality \eqref{wkSol energy}.
\begin{lemma}\label{lem:Hom3d energybds}
    Let the hypotheses and notations of \Cref{thm:Hom} with $d=3$ hold true.
    Then we have
    \begin{equation}\label{Hom3D energybds}
        \begin{aligned}
            &\left\|\sqrt{\rhoeps}\uveceps\right\|_{L^\infty(0,T;L^2(\Omegaeps))}
            +
            \|\rhoeps \|_{L^\infty(0,T;L^\gamma(\Omegaeps))}
            +
            \|\ceps\|_{L^\infty(0,T;W^{1,2}(\Omegaeps))}
            \\
            &\quad
            +
            \|\uveceps\|_{L^2(0,T;W^{1,2}(\Omegaeps))}
            +
            \|\partial_t \ceps\|_{L^2((0,T)\times\Omegaeps)}
            \leq 
            C.
        \end{aligned}
    \end{equation}
\end{lemma}
\begin{proof}
    The energy inequality \eqref{wkSol energy} implies for almost all $\tau \in (0,T)$ that
    \begin{equation*}
    \begin{aligned}
        &\int_{\Omegaeps}
        \frac{1}{2}\rhoeps(\tau,\cdot)|\uveceps(\tau,\cdot)|^2 + W(\rhoeps(\tau,\cdot)) + \frac{\coup}{2}|\rhoeps(\tau,\cdot)-\ceps(\tau,\cdot)|^2 + \frac{\kappa}{2}|\nablax \ceps(\tau,\cdot)|^2 
        \, \dd x 
        \\
        &\quad +
        \int_0^\tau 
        \int_{\Omegaeps} \Svisc(\nablax\uveceps):\nablax\uveceps + \paracoup|\partial_t\ceps|^2 
        \, \dd x \, \dd t
        \\
        &\leq 
        \int_{\Omegaeps}
        \frac{\big|(\rho\uvec)_{0\eps}\big|^2}{2\rho_{0\eps}} + W(\rho_{0\eps})+\frac{\coup}{2}|\rho_{0\eps}-c_{0\eps}|^2 + \frac{\kappa}{2} |\nablax c_{0\eps}|^2 
        \, \dd x 
        \leq C,
    \end{aligned}
    \end{equation*}
    where we have used the assumptions on the initial data in \eqref{thm:Hom IC} to obtain the last inequality.
    From this inequality, we deduce by using \eqref{dec prs reg}--\eqref{dec prs grwth} that
    \begin{equation*}
        \begin{aligned}
            &\left\|\sqrt{\rhoeps}\uveceps\right\|_{L^\infty(0,T;L^2(\Omegaeps))}
            +
            \|\rhoeps\|_{L^\infty(0,T;L^\gamma(\Omegaeps))}
            +
            \|\rhoeps-\ceps\|_{L^\infty(0,T;L^2(\Omegaeps))}
            \\
            &\quad 
            +
            \|\nablax \ceps \|_{L^\infty(0,T;L^2(\Omegaeps))}
            +
            \|\nablax\uveceps\|_{L^2((0,T)\times\Omegaeps)}
            +
            \|\partial_t \ceps\|_{L^2((0,T)\times\Omegaeps)}
            \leq C.
        \end{aligned}
    \end{equation*}
    Since $\gamma>3$, we infer
    \begin{equation*}
        \|\ceps\|_{L^\infty(0,T;L^2(\Omegaeps))}
        \leq 
        \|\rhoeps-\ceps\|_{L^\infty(0,T;L^2(\Omegaeps))}
        +
        \|\rhoeps\|_{L^\infty(0,T;L^2(\Omegaeps))}
        \leq 
        C.
    \end{equation*}
    From Poincar\'e's inequality on $\Omega$, we conclude
    \begin{equation*}
        \|\uveceps\|_{L^2((0,T)\times\Omegaeps)}
        =
        \|\tuveceps\|_{L^2((0,T)\times\Omega)}
        \leq 
        C
        \|\nablax \tuveceps\|_{L^2((0,T)\times\Omega)}
        =
        C\|\nablax \uveceps\|_{L^2((0,T)\times\Omegaeps)}
        \leq C.
    \end{equation*}
\end{proof}
To obtain compactness for the pressure term $\Peff(\rhoeps)$, we need an improved uniform bound for the density $\rhoeps$.
For the compressible NSE, such a uniform bound has been obtained by testing the momentum equation with a specific function, which is defined by means of the Bogovski\u{\i} operator $\Bogeps$ on the porous domain $\Omegaeps$ constructed in \cite{LuSchwarzacher2018,DieningFeireislLu2017}.
For the sake of completeness we state here the properties of this operator as a lemma.
\begin{lemma}\label{lem:Hom3D Bogeps}
For $\eps \in (0,1)$, let $\Omega_\eps$ be defined as in \eqref{def:domain} with $d=3$ and let $s \in (1,\infty)$. Then there exists a bounded linear operator $\mathcal{B}_\eps : L_0^s(\Omega_\eps) \to W_0^{1,s}(\Omega_\eps; \RR^3)$ such that for any $f \in L_0^s(\Omega_\eps)$,
\begin{align*}
    \div \mathcal{B}_\eps(f) = f \text{ a.e.~in } \Omega_\eps,
    \quad
    \|\mathcal{B}_\eps(f)\|_{W_0^{1,s}(\Omega_\eps)} \leq C \Big( 1 + \eps^\frac{(3-s)\alpha - 3}{s} \Big) \|f\|_{L^s(\Omega_\eps)}.
\end{align*}
Moreover, for any $s>\frac32$, this operator can be extended to a linear operator
\begin{align*}
    \mathcal{B}_\eps\colon \big[\dot{W}^{1,s^\prime}(\Omegaeps)\big]^\ast :=  \big\{ f \in \big[W^{1, s'}(\Omega_\eps)\big]^\ast : \langle f, 1\rangle=0 \big\} \to L^s(\Omega_\eps;\RR^3)
\end{align*}
such that, for any $f\in \big[\dot{W}^{1,s^\prime}(\Omegaeps)\big]^\ast$,
\begin{align*}
    -\int_{\Omegaeps}\Bogeps(f)\cdot \nablax \phi\, \dd x
    = \langle f,  \phi\rangle \quad \forall\,\phi\in W^{1,s'}(\Omega_\eps), 
    \quad
    \|\mathcal{B}_\eps (f) \|_{L^s(\Omega_\eps)} \leq C \|f\|_{[W^{1,s^\prime}(\Omega_\eps)]^\ast}.
\end{align*}
\begin{proof}
    See \cite[Theorem~2.3]{DieningFeireislLu2017} and \cite[Proposition~2.3]{LuSchwarzacher2018}.
\end{proof}
\end{lemma}
By interpreting the additional term $\rhoeps\nablax \ceps$ as a force term and using the uniform bounds in Lemma~\ref{lem:Hom3d energybds}, we adapt the technique from \cite{OschmannPokorny2023} to the rNSKE and obtain the following improved bound for the density.
\providecommand{\Boge}{\mathcal{B}_\eps}
\begin{lemma}\label{lem:impr prs Hom3D}
    Let the hypotheses and notations of \Cref{thm:Hom} with $d=3$ hold true.
    Then we have
    \begin{equation}\label{impr prs Hom3D}
        \left\|\rhoeps\right\|_{L^{\frac{5\gamma-3}{3}}((0,T)\times \Omegaeps)}
        +
        \left\|\Peff(\rhoeps)\right\|_{L^{\frac{5\gamma-3}{3\gamma}}((0,T)\times \Omegaeps)} 
        \leq 
        C.
    \end{equation}
\end{lemma}
\begin{proof}
    Let us first assume that
    \begin{equation}\label{add regularity}
        \rhoeps^{\theta},\,\,\partial_t(\rhoeps^\theta) \in L^\infty(0,T;L^\infty(\Omegaeps)).
    \end{equation}
    We define
    \begin{align*}
        \theta:=\frac{2\gamma-3}{3},
        &&
        \varphivec_{\eps} := \Bogeps\left(\rhoeps^\theta - \fint_{\Omegaeps}\rhoeps^\theta \right)
    \end{align*}
    where $\Bogeps$ denotes the Bogovski\u{\i} operator on the porous domain $\Omegaeps$ (see \Cref{lem:Hom3D Bogeps}).
    Thanks to \eqref{wkSol reg} and \eqref{wkSol cont} we have
    \begin{equation*}
        \partial_t(\trhoeps^\theta) + \div(\trhoeps^\theta\tuveceps) = (1-\theta) \trhoeps^\theta \div\tuveceps
        \quad \text{in } \mathcal{D}^\prime(\RR_T^3),
    \end{equation*}
    see e.g.~\cite[Lemma~11.13]{FeireislNovotny2017singlim}.
    Using the regularity \eqref{add regularity} this implies $\div(\rhoeps^\theta\uveceps)\in L^2((0,T)\times\Omegaeps)$ and thus, $\rhoeps^\theta\uveceps$ has a well-defined normal trace on $\partial\Omegaeps$.
    In particular, we have that the latter identity holds almost everywhere on $\Omegaeps$, which implies
    \begin{equation}\label{norm trace}
        \rhoeps^\theta\uveceps(\tau)\cdot\mathbf{n}_{\partial\Omegaeps} = 0
        \quad \text{in the sense of traces for almost all } \tau \in (0,T),
    \end{equation}
    and
    \begin{equation}\label{deriv test fct Bogeps}
        \partial_t\varphivec_\eps 
        = 
        -\Bogeps\big(\div(\rhoeps^\theta\uveceps)\big) 
        +
        (1-\theta) \Bogeps\big(\rhoeps^\theta \div\uveceps - \fint_{\Omegaeps}\rhoeps^\theta \div\uveceps\big)
        \quad \text{a.e.~in } (0,T)\times\Omegaeps.
    \end{equation}
    We fix $\psi \in C^\infty_c((0,T))$ with $\psi \geq 0$.
    Using a density argument we can use $\psi\varphivec_\eps$ as a test function in the momentum equation and obtain
    \begin{equation*}
        \begin{aligned}
            & \int_0^T \psi \int_{\Omega_\eps} 
            \Peff(\rhoeps) \rho_\eps^\theta 
            \,\dx\,\dt 
            = 
            \int_0^T \psi 
            \bigg( \int_{\Omega_\eps} \Peff(\rhoeps) \rho_\eps^\theta \dx \bigg)
            \bigg( \fint_{\Omega_\eps} \rho_\eps^{\theta} \bigg) 
            \dt   
            -\int_0^T \int_{\Omega_\eps} 
            \partial_t \psi \,\rho_\eps \uvec_\eps \cdot \varphivec_\eps 
            \,\dx \,\dt 
            \\
            &+  
            \int_0^T 
            \int_{\Omega_\eps} \psi \Svisc(\nablax \uvec_\eps):\nablax \varphivec_\eps 
            \, \dx\,  \dt 
            - 
            \int_0^T \int_{\Omega_\eps} 
            \psi \rho_\eps \uvec_\eps\otimes \uvec_\eps : \nablax \varphivec_\eps
            \dx \dt
            -
            \int_0^T \int_{\Omega_\eps} \psi \rho_\eps \uvec_\eps \cdot \partial_t \varphivec_\eps 
            \dx \dt 
            \\
            &- 
            \int_0^T \int_{\Omega_\eps} 
            \coup \psi \rho_\eps \nablax c_\eps \cdot \varphivec_\eps \dx \dt 
            =: \sum\limits_{i=1}^{6}I_i^\eps.
        \end{aligned}
    \end{equation*}
    Using \eqref{thm:Hom3D restr gamma,alpha}, \eqref{norm trace}, \eqref{deriv test fct Bogeps}, \Cref{lem:Hom3d energybds}, and \Cref{lem:Hom3D Bogeps}, we estimate the terms $I_i^\eps$ for $i \in \{1,2,\dots,5\}$ as
    \begin{equation*}
        \sum\limits_{i=1}^5 \left|I_i^\eps\right|
        \leq 
        C \left(1+ \| \psi\|_{L^\infty((0,T))} + \|\partial_t\psi\|_{L^1((0,T))}\right)
        + 
        \frac{3}{4} \int_0^T \psi \int_{\Omegaeps} 
        \rhoeps^{\gamma+\theta}
        \, \dx \, \dt.
    \end{equation*}
    Since these estimates where performed in \cite[Section~3.2]{OschmannPokorny2023}, we omit the details here and refer the reader to the mentioned reference for more details.
    For the last term $I_6^\eps$, we use  \Cref{lem:Hom3d energybds}, \Cref{lem:Hom3D Bogeps}, and the Sobolev embedding $W^{1,{\frac{6\gamma}{5\gamma - 6}}}(\Omega)\hookrightarrow L^{\frac{2\gamma}{\gamma-2}}(\Omega)$ to obtain
    \begin{equation*}
    \begin{aligned}
        |I_6^\eps|
        &\leq 
        C\|\psi\|_{L^\infty((0,T))}\|\rhoeps\nablax\ceps\|_{L^\infty(0,T;L^{\frac{2\gamma}{\gamma+2}}(\Omegaeps))} \|\varphivec_\eps\|_{L^1(0,T;L^{\frac{2\gamma}{\gamma-2}}(\Omegaeps))}
        \\
        &\leq 
        C\|\psi\|_{L^\infty((0,T))} \|\rhoeps\|^\theta_{L^1(0,T;L^{\frac{6\gamma\theta}{5\gamma-6}}(\Omegaeps))}
        \leq C\|\psi\|_{L^\infty((0,T))}.
    \end{aligned}
    \end{equation*}
    In the second inequality, we have used the fact that the inequality
    \begin{equation*}
        \left(3 - \frac{6\gamma}{5\gamma-6}\right) \alpha - 3 \geq 0 
        \Longleftrightarrow \alpha \geq \frac{5\gamma-6}{3(\gamma-2)}
    \end{equation*}
    holds due to assumption \eqref{thm:Hom3D restr gamma,alpha}, since
    \begin{equation*}
        \frac{2\gamma-3}{\gamma-3} \geq \frac{5\gamma-6}{3(\gamma-2)} \qquad \forall\, \gamma \in (3,\infty).
    \end{equation*}
    In the third inequality, we have used \Cref{lem:Hom3d energybds} and the fact that
    \begin{equation*}
        \frac{6\gamma \theta}{5\gamma -6} 
        =
        \frac{\gamma(4\gamma-6)}{5\gamma-6}
        \leq \gamma 
        \qquad \forall \, \gamma \in \left(\frac{6}{5},\infty\right).
    \end{equation*}
    On the other hand, since $p$ satisfies \eqref{dec prs reg}--\eqref{dec prs grwth}, we have that
    \begin{equation*}
        \int_0^T \psi \int_\Omega 
        \rhoeps^{\gamma + \theta} 
        \, \dx \, \dt 
        \leq 
        C + C \int_0^T \psi \int_\Omega 
        \Peff(\rhoeps) \rhoeps^\theta 
        \, \dx \, \dt .
    \end{equation*}
    In total, we have shown for any $\psi \in C^\infty_c((0,T))$ satisfying $\psi\geq 0$ that
    \begin{equation*}
        \int_0^T \psi\int_\Omega \rhoeps^{\gamma+\theta} \, \dx \, \dt 
        \leq 
        C\left(1+ \|\psi\|_{L^\infty((0,T))} + \|\partial_t\psi\|_{L^1((0,T))}\right),
    \end{equation*}
    where the positive constant $C$ does not depend on $\psi$.
    Approximating the identity function on $(0,T)$ with test functions from $\mathcal{D}((0,T))$ yields the estimate for the first term in \eqref{impr prs Hom3D}.
    The estimate for the second term in \eqref{impr prs Hom3D} follows from the fact that $p$ satisfies \eqref{dec prs reg}--\eqref{dec prs grwth}.
    To get rid of the assumption \eqref{add regularity} we use the regularization from \cite[Section~7.9.5]{NovotnyStraskraba2004}.
\end{proof}
\subsection{Extensions of functions}\label{subsec:extsfunct3D}
To perform the homogenization limit $\eps \to 0$, we first extend the weak solution $(\rhoeps,\uveceps,\ceps)$ to the $\eps$-independent domain $\OmegaT$ by means of suitable extensions.
More specifically, as $\rhoeps \in C_{\mathrm{w}}(0,T;L^\gamma(\Omegaeps))$ and $\uveceps \in L^2(0,T;W^{1,2}_0(\Omegaeps))$, we use for both quantities the zero extension preserving the corresponding regularities $\trhoeps \in C_{\mathrm{w}}(0,T;L^\gamma(\Omega))$ and $\tuveceps\in L^2(0,T;W^{1,2}_0(\Omega))$.
For the relaxation parameter $\ceps$, we cannot choose this extension to obtain that the same regularity in $L^\infty(0,T;W^{1,2}(\Omega))$, which is due to the boundary conditions $\nablax \ceps \cdot \mathbf{n} = 0$ on $\partial \Omegaeps$\footnote{This is similar to extending the temperature in Navier-Stokes-Fourier equations, see e.g. \cite{LuPokorny2021, Oschmann2022, PokornySkrisovsky2021a}}.
Therefore, we choose the extension operator $\Eeps$ from Lemma~\ref{lem:extEeps}.
With this choice, we have that the corresponding extension indeed satisfies $\Eeps\ceps\in L^\infty(0,T;W^{1,2}(\Omega))$.
From Lemma~\ref{lem:Hom3d energybds} and Lemma~\ref{lem:impr prs Hom3D}, we further immediately conclude the following uniform bounds for the extended quantities $\trhoeps,\tuveceps,\tceps,\Eeps\ceps$.
\begin{lemma}\label{lem:Hom3d bds exts}
    Let the hypotheses and notations of \Cref{thm:Hom} with $d=3$ hold true.
    Then we have 
    \begin{equation}\label{Hom3d bds exts I}
        \begin{aligned}
            &\|\sqrt{\trhoeps}\tuveceps\|_{L^\infty(0,T;L^{2}(\Omega))}
            +
            \|\trhoeps\|_{L^\infty(0,T;L^\gamma(\Omega))}
            +
            \|\Eeps\ceps\|_{L^\infty(0,T;W^{1,2}(\Omega))}
            +
            \|\tceps\|_{L^\infty(0,T;L^2(\Omega))}
            \\
            &\quad +
            \|\tuveceps\|_{L^2(0,T;W^{1,2}(\Omega))}
            +
            \|\partial_t \Eeps\ceps\|_{L^2(\OmegaT)}
            +
            \|\partial_t\tceps\|_{L^2(\OmegaT)}
            +
            \|\trhoeps\|_{L^{\frac{5\gamma-3}{3}}(\OmegaT)}
            \\
            &\quad 
            +
            \|\Peff(\trhoeps)\|_{L^{\frac{5\gamma-3}{3\gamma}}(\OmegaT)}
            \leq 
            C.
        \end{aligned}
    \end{equation}
    Moreover, we have
    \begin{equation}\label{Hom3d bds exts II}
    \begin{aligned}
        &\|\Eeps\ceps-\tilde{c}_\eps\|_{L^\infty(0,T;L^2(\Omega))}
        =
        \|\Eeps\ceps\|_{L^\infty(0,T;L^{2}(\Omega\setminus\Omegaeps))}
        \leq 
        C\eps^{\alpha-1},
        \\
        &\|\nablax \Eeps\ceps\|_{L^\infty(0,T;L^1(\Omega\setminus\Omegaeps))}
        \leq 
        C \eps^{\frac{3(\alpha-1)}{2}}.
    \end{aligned}
    \end{equation}
\end{lemma}
\begin{proof}
    The bounds from \eqref{Hom3d bds exts I} follow from Lemma~\ref{lem:Hom3d energybds} and Lemma~\ref{lem:impr prs Hom3D} by using the definition of the zero extension and the properties of the extension operator $\Eeps$ in Lemma~\ref{lem:extEeps}.
    In view of \eqref{porous dom number density}, we infer by using again \Cref{lem:extEeps}, the bounds in \eqref{Hom3d bds exts I}, H\"older's inequality, and the Sobolev embedding $W^{1,2}(\Omega) \hookrightarrow L^{6}(\Omega)$ that
    \begin{equation*}
    \begin{aligned}
        &\|\Eeps\ceps\|_{L^\infty(0,T;L^2(\Omega\setminus\Omegaeps))}
        \leq 
        \|\Eeps\ceps\|_{L^\infty(0,T;L^6(\Omega))}
        \left|\Omega\setminus\Omegaeps\right|^{\frac{1}{3}}
        \leq 
        C\eps^{\alpha-1},
        \\
        &\|\nablax\Eeps\ceps\|_{L^\infty(0,T;L^1(\Omega\setminus\Omegaeps))}
        \leq 
        \|\nablax\Eeps\ceps\|_{L^\infty(0,T;L^2(\Omega))}|\Omega\setminus\Omegaeps|^{\frac{1}{2}}
        \leq 
        C\eps^{\frac{3(\alpha-1)}{2}}.
    \end{aligned}
    \end{equation*}
\end{proof}
For the extended quantities $\trhoeps,\tuveceps,\tilde{c}_\eps,\Eeps\ceps$ we immediately verify the following versions of the continuity equation and the parabolic equation in the unperforated domain $\OmegaT$.
\begin{lemma}\label{lem:Hom3d cont+parab exts}
    Let the hypotheses and notations of \Cref{thm:Hom} with $d=3$ hold true. Then we have for any $\varphi\in \mathcal{D}(\RR^3_T)$ that
    \begin{align}
            &\int_0^T \int_\Omega
            \trhoeps \partial_t\varphi + \trhoeps\tuveceps \cdot \nablax \varphi 
            \, \dd x \, \dd t = 0,\label{Hom3d cont exts}
            \\
            &\int_0^T \int_\Omega 
            \paracoup\tilde{c}_\eps \partial_t \varphi - \kappa \nablax \Eeps\ceps \cdot \nablax \varphi 
            -
            \coup (\tilde{c}_{\eps} - \trhoeps)\varphi 
            \, \dd x \, \dd t
            = \langle \mathcal{F}_{\eps},  \varphi\rangle\label{Hom3d parab exts},
    \end{align}
    where
    \begin{equation*}
        \langle\mathcal{F}_\eps,\varphi\rangle
        :=
        -
        \int_0^T \int_{\Omega\setminus\Omegaeps} 
        \kappa \nablax \Eeps\ceps \cdot \nablax \varphi 
        \, \dd x \, \dd t
        \xrightarrow{\eps\to 0}0.
    \end{equation*}
\end{lemma}
\begin{proof}
    Equations \eqref{Hom3d cont exts} and \eqref{Hom3d parab exts} follow from the fact that the weak solution $(\rhoeps,\uveceps,\ceps)$ satisfies the weak formulations of the continuity equation \eqref{wkSol cont} and the parabolic equation \eqref{wkSol parab} on the perforated domain $(0,T)\times \Omegaeps$ by using the definition of the zero extension and the properties of the extension operator $\Eeps$ in Lemma~\ref{lem:extEeps}.
    With \eqref{Hom3d bds exts II} we estimate
    \begin{equation*}
        |\langle\mathcal{F}_\eps,\varphi\rangle|
        \leq 
        C\|\varphi\|_{W^{1,\infty}(\OmegaT)} \|\nablax \Eeps\ceps\|_{L^\infty(0,T;L^1(\Omega\setminus\Omegaeps))}
        \leq 
        C\|\varphi\|_{W^{1,\infty}(\OmegaT)}\eps^{\frac{3(\alpha-1)}{2}}\xrightarrow{\eps \to 0} 0.
    \end{equation*}
\end{proof}
To derive a corresponding momentum equation on the unperforated domain $\OmegaT$, we wish to use a function $\varphivec \in \mathcal{D}(\OmegaT;\RR^3)$ as a test function in the weak formulation of the momentum equation \eqref{wkSol mom}.
However, this is in general not possible due to $\varphivec\neq 0$ on $(0,T)\times(\Omega \setminus\Omegaeps)$.
To circumvent this issue, we use the following lemma introducing a well-behaved cut-off (see \cite{Bravin2024,OschmannPokorny2023}).
\begin{lemma}\label{lemPhi}
Assume that $\eps>0$ and $\alpha>3$ satisfy $2\eps^{\alpha}<\eps^{3}$ and let $\Omegaeps$ be defined as in \eqref{def:domain}.
Then there exists a solenoidal matrix-valued function $\Phi_\eps$ such that
\begin{align*}
\Phi_\eps &\in W^{1,\infty^-}(\Omega; \RR^{3\times 3})\cap L^\infty(\Omega; \RR^{3\times 3}),
\quad %\\,% \quad \forall \, s \in [1,\infty),\\
\Phi_\eps =0 \text{ on } \Omega\setminus \Omega_\eps,
\quad 
\Phi_\eps =\mathbb{I} \text{ on } \Omega\setminus \bigcup_{i \in K_\eps} B_{\eps^{3}}(\eps x_i).
\end{align*}
Moreover, $\|\Phi_\eps\|_{L^\infty(\Omega)}\leq C$, and for any $1\leq s< \infty$,
\begin{align*}
\|\Phi_\eps-\mathbb{I}\|_{L^s(\Omega)}^s &\leq C \eps^{6}, &&\|\nabla\Phi_\eps\|_{L^s(\Omega)}^s \leq C \eps^{(3-s)\alpha - 3} \begin{cases}
|\eps^{(3-2s)(3-\alpha)} - 1| & \text{if } s \neq \frac32,\\
|\log(\eps^{3-\alpha})| & \text{if } s=\frac32.
\end{cases}
\end{align*}
In turn, for any $\varphivec \in C_c^\infty(\Omega;\RR^3)$ and any $s>\frac32$,
\begin{align*}
\|\nabla (\Phi_\eps \varphivec)-\Phi_\eps \nabla\varphivec\|_{L^s(\Omega)}^s &\leq C \eps^{(3-s)\alpha - 3} \|\varphivec\|_{L^\infty(\Omega)}^s.
\end{align*}
\end{lemma}
\begin{proof}
    The statement is precisely the case $\delta=2$ in \cite[Lemma~6.1]{OschmannPokorny2023}.
\end{proof}
With the cut-off function $\Phi_\eps$ at hand we derive for the extended quantities $\trhoeps,\tuveceps,\tceps,\Eeps\ceps$ the following version of the momentum equation on the unperforated domain $\OmegaT$.
\begin{lemma}\label{lem:Hom3d mom exts}
    Let the hypotheses and notations of \Cref{thm:Hom} with $d=3$ hold true and let $\eps_0\in (0,1)$ satisfy $2\eps_0^\alpha<\eps_0^3$.
    For any $\eps \in (0,\eps_0)$, let $\Phi_\eps$ denote the cut-off function from Lemma~\ref{lemPhi}.
    Then we have for any $\varphivec\in 
    \mathcal{D}(\OmegaT;\RR^3)$ that
    \begin{equation}\label{Hom3d mom exts}
        \begin{aligned}
            &\int_0^T \int_\Omega 
            \Phi_\eps^{\mathrm{T}}\trhoeps\tuveceps \cdot \partial_t \varphivec 
            +
            \Phi_\eps^{\mathrm{T}}\trhoeps\tuveceps\otimes \tuveceps : \nablax \varphivec
            +
            \Peff(\trhoeps)\div\varphivec
            \, \dd x \, \dd t
            \\
            &-
            \int_0^T \int_\Omega 
            \Svisc(\nablax\tuveceps):\nablax\varphivec - \coup \trhoeps \nablax \Eeps\ceps \cdot \varphivec
            \, \dd x \, \dd t
            =
            \langle\mathcal{G}_\eps,\varphivec\rangle
        \end{aligned}
    \end{equation}
    where
    \begin{equation*}
        \begin{aligned}
            \langle\mathcal{G}_\eps,\varphivec\rangle
            &:=
            \int_0^T \int_\Omega 
            \Big(
            \trhoeps\tuveceps\otimes\tuveceps:\big(\Phi_\eps\nablax\varphivec - \nablax(\Phi_\eps\varphivec)\big)
            +
            \Peff(\trhoeps)\big(\mathbb{I}-\Phi_\eps\bigr):\nablax\varphivec
            \Big)
            \, \dd x \, \dd t
            \\
            &\quad +
            \int_0^T \int_\Omega 
            \Svisc(\nablax\tuveceps):\Big(\big(\Phi_\eps - \mathbb{I}\big)\nablax\varphivec+\big(\nablax(\Phi_\eps\varphivec)-\Phi_\eps\nablax\varphivec\big)\Big)
            \, \dd x \, \dd t
            \\
            &\quad +
            \int_0^T \int_\Omega 
            \coup\trhoeps\nablax\Eeps\ceps \big(\mathbb{I}-\Phi_\eps\big)\varphivec\, \dd x \, \dd t.
        \end{aligned}
    \end{equation*}
    Furthermore, there exists some $\delta \in(0,\infty)$, such that
    \begin{equation}\label{Hom3d est remainder G_eps}
        |\langle\mathcal{G}_\eps,\varphivec\rangle|\leq C\eps^\delta \|\varphivec\|_{L^{\frac{5\gamma-3}{3}}(0,T;W^{1,\frac{5\gamma-3}{3}}(\Omega))}
        \quad \forall\, \varphivec\in \mathcal{D}(\OmegaT;\RR^3).
    \end{equation}
    In particular, we have 
    \begin{equation}\label{Hom3d convergence remainder G_eps}
        \mathcal{G}_\eps \xrightarrow{\eps\to 0} 0 \quad \text{in } L^{\frac{5\gamma-3}{5\gamma-6}}(0,T;W^{-1,\frac{5\gamma-3}{5\gamma-6}}(\Omega;\RR^3)).
    \end{equation}
\end{lemma}
\begin{proof}
    Let $\varphivec\in \mathcal{D}(\OmegaT;\RR^3)$.
    Then we have that $\Phi_\eps\varphivec \in C^\infty_c((0,T);W^{1,\infty^-}_0(\Omegaeps;\RR^3))$.
    By a density argument, $\Phi_\eps \varphivec$ is a valid test function for the momentum equation \eqref{wkSol mom} on the porous domain $(0,T)\times\Omegaeps$ which is satisfied by the weak solution $(\rhoeps,\uveceps,\ceps)$.
    We obtain
    \begin{equation*}
        \begin{aligned}
            &\int_0^T \int_{\Omegaeps}
            \rhoeps\uveceps \cdot \partial_t(\Phi_\eps\varphivec)
            +
            \rhoeps\uveceps\otimes\uveceps : \nablax(\Phi_\eps\varphivec)
            +
            \Peff(\rhoeps)\div(\Phi_\eps\varphivec)
            \, \dd x \, \dd t
            \\
            &-
            \int_0^T \int_{\Omegaeps}
            \Svisc(\nablax\uveceps):\nablax(\Phi_\eps\varphivec)-\coup\rhoeps\nablax \ceps\cdot  \Phi_\eps \varphivec
            \, \dd x \, \dd t 
            = 0.
        \end{aligned}
    \end{equation*}
    Using that $\Phi_\eps=0$, $\trhoeps=0$, $\tuveceps=0$ on $\Omega\setminus\Omegaeps$, $\Eeps\ceps = \ceps$ on $\Omegaeps$, as well as $\div\Phi_\eps=0$ on $\Omega$, we may rewrite this relation as
    \begin{equation*}
        \begin{aligned}
            &\int_0^T \int_{\Omega}
            \Phi_\eps^{\mathrm{T}}\trhoeps\tuveceps \cdot \partial_t\varphivec
            +
            \trhoeps\tuveceps\otimes\tuveceps : \nablax(\Phi_\eps\varphivec)
            +
            \Peff(\trhoeps)\Phi_\eps:\nablax\varphivec
            \, \dd x \, \dd t
            \\
            &-
            \int_0^T \int_{\Omega}
            \Svisc(\nablax\tuveceps):\nablax(\Phi_\eps\varphivec)-\coup\trhoeps\nablax \Eeps\ceps\cdot  \Phi_\eps \varphivec
            \, \dd x \, \dd t 
            = 0.
        \end{aligned}
    \end{equation*}
    Rearranging this equation leads to \eqref{Hom3d mom exts}.
    Since we have
    \begin{align*}
            &\frac{3}{5\gamma-3} + \frac{3\gamma}{5\gamma-3} < 1
            \quad \forall\, \gamma \in (3,\infty),
            &&
            \frac{3}{5\gamma-3} + \frac{1}{2} <1 
            \quad \forall \gamma \in (1,\infty),
    \end{align*}
    and
    \begin{equation*}
        \frac{5\gamma-3}{3}>3 \quad \forall\, \gamma \in \left(\frac{12}{5},\infty\right),
    \end{equation*}
    we can estimate with Lemma~\ref{lem:Hom3d bds exts}, Lemma~\ref{lemPhi}, the Sobolev embedding $W^{1,\frac{5\gamma-3}{3}}(\Omega)\hookrightarrow L^\infty(\Omega)$, as well as H\"older's inequality,
    \begin{equation*}
        \begin{aligned}
            |\langle \mathcal{G}_\eps,\varphivec\rangle|
            &\leq 
            C\|\trhoeps\tuveceps\otimes\tuveceps\|_{L^2(0,T;L^{\frac{3\gamma}{\gamma+3}}(\Omega))} \|\Phi_\eps\nablax\varphivec-\nablax(\Phi_\eps\varphivec)\|_{L^2(0,T;L^{\frac{3\gamma}{2\gamma-3}}(\Omega))}
            \\
            &\quad+
            C\|\Peff(\trhoeps)\|_{L^{\frac{5\gamma-3}{3\gamma}}(\OmegaT)} \|\Phi_\eps-\mathbb{I}\|_{L^s(\Omega)} \|\nablax\varphivec\|_{L^{\frac{5\gamma-3}{3}}(\OmegaT)}
            \\
            &\quad +
            C\|\Svisc(\nablax\tuveceps)\|_{L^2(\OmegaT)} 
            \Big(
            \|\Phi_\eps-\mathbb{I}\|_{L^s(\Omega)}\|\nablax\varphivec\|_{L^{\frac{5\gamma-3}{3}}(\OmegaT)} 
            +
            \|\Phi_\eps\nablax\varphivec - \nablax(\Phi_\eps\mathbb{I})\|_{L^2(\OmegaT)}
            \Big)
            \\
            &\quad +
            C\|\trhoeps \nablax\Eeps\ceps\|_{L^\infty(0,T;L^{\frac{2\gamma}{\gamma+2}}(\Omega))} \|\Phi_\eps-\mathbb{I}\|_{L^s(\Omega)}\|\varphivec\|_{L^{\frac{5\gamma-3}{3}}(0,T; L^\infty(\Omega))}
            \\
            &\leq 
            C\Big(
            \|\Phi_\eps\nablax\varphivec - \nablax(\Phi_\eps\varphivec)\|_{L^2(0,T;L^{\frac{3\gamma}{2\gamma-3}}(\Omega))} 
            +
            \|\Phi_\eps\nablax\varphivec - \nablax(\Phi_\eps\varphivec)\|_{L^2(\OmegaT)}
            \Bigr)
            \\
            &\quad +
            C\|\Phi_\eps-\mathbb{I}\|_{L^s(\Omega)}\|\varphivec\|_{L^{\frac{5\gamma-3}{3}}(0,T;W^{1,\frac{5\gamma-3}{3}}(\Omega))}
            \\
            &\leq 
            C\Big(
            \eps^{\frac{2\gamma-3}{3\gamma}\big((3-\frac{3\gamma}{2\gamma-3})\alpha-3\big)} + \eps^{\frac{\alpha-3}{2}}
            \Big)
            \|\varphivec\|_{L^2(0,T;L^\infty(\Omega))}
            +
            \eps^{\frac{6}{s}}\|\varphivec\|_{L^{\frac{5\gamma-3}{3}}(0,T;W^{1,\frac{5\gamma-3}{3}}(\Omega))}
            \\
            &\leq 
            C\eps^\delta \|\varphivec\|_{L^{\frac{5\gamma-3}{3}}(0,T;W^{1,\frac{5\gamma-3}{3}}(\Omega))},
        \end{aligned}
    \end{equation*}
    with
    \begin{align*}
        s = \max \left \{ \frac{5\gamma-3}{2(\gamma-3)}, \frac{2(5\gamma-3)}{5\gamma-9}, \frac{2\gamma}{\gamma-2} \right\} = \frac{5\gamma-3}{2(\gamma-3)},
    \end{align*}
    and
    \begin{equation*}
        \delta:=
        \min\left\{
        \frac{2\gamma-3}{3\gamma}\bigg(\Big(3-\frac{3\gamma}{2\gamma-3}\Big)\alpha-3\bigg), \frac{\alpha-3}{2}, \frac{6}{s}
        \right\}.
    \end{equation*}
    In view of the assumptions \eqref{thm:Hom3D restr gamma,alpha} we have $\delta>0$.
    Thus, we conclude \eqref{Hom3d est remainder G_eps}, and, in particular,
    \begin{equation*}
        \mathcal{G}_\eps \xrightarrow{\eps \to 0} 0
        \quad \text{in } \left(L^{\frac{5\gamma-3}{3}}(0,T;W^{1,\frac{5\gamma-3}{3}}_0(\Omega;\RR^3))\right)^* \simeq  L^{\frac{5\gamma-3}{5\gamma-6}}(0,T;W^{-1,\frac{5\gamma-3}{5\gamma-6}}(\Omega;\RR^3)).
    \end{equation*}
\end{proof}
\begin{remark}
    It is crucial that the second relation in \eqref{thm:Hom3D restr gamma,alpha} is fulfilled with a strict inequality.
    In the case $\alpha = \max\big\{3,\frac{2\gamma-3}{\gamma-3}\big\}$, we would only obtain $\delta = 0$ in the proof of \Cref{lem:Hom3d mom exts}, such that we cannot show $\mathcal{G}_\eps \to 0$. Indeed this is even not expected: in the incompressible case corresponding to (roughly speaking) $\gamma = \infty$, the case $\alpha=3$ leads to an additional Brinkman term; see \cite{Allaire1990a, BellaOschmann2022} for details. Unfortunately, in the fully compressible case, this question is completely open.
\end{remark}

\subsection{Homogenization limit}\label{subsec:Hom3DLim}
In this subsection, we perform the homogenization limit $\eps \to 0$ and complete the proof of \Cref{thm:Hom}.
The uniform bounds for the extended quantities $\trhoeps,\tuveceps,\tceps,\Eeps\ceps$ are strong enough such that we can exploit compactness arguments that allow us, up to a subsequence, to pass to the limit $\eps \to 0$ in the weak formulations $\eqref{Hom3d cont exts}$, $\eqref{Hom3d parab exts}$, and $\eqref{Hom3d mom exts}$.
More specifically, we have the following statement.
\begin{lemma}\label{lem:Hom3d weak lim}
    Let the hypotheses and notations of \Cref{thm:Hom} with $d=3$ hold true.
    Assume that $\eps_0>0$ satisfies $2\eps_0^\alpha<\eps_0^3$ and let $\Phi_\eps$ denote the cut-off function from Lemma~\ref{lemPhi} for any $\eps \in (0,\eps_0)$.
    Then we have, after passing to a non-relabeled subsequence,
    \begin{equation}\label{Hom3d weak limits}
        \begin{aligned}
            &\trhoeps \weak \rho \quad \text{in } L^{\frac{5\gamma-3}{3}}(\OmegaT),
            \quad 
            \trhoeps \to \rho \quad \text{in } C_{\mathrm{w}}([0,T];L^\gamma(\Omega)),
            \\
            &\tuveceps \weak \uvec \quad \text{in } L^2(0,T;W^{1,2}_0(\Omega)),
            \quad
            \Phi^\mathrm{T}_\eps \trhoeps\tuveceps \to \rho\uvec \quad \text{in } C_{\mathrm{w}}([0,T];L^{\frac{2\gamma}{\gamma+1}}(\Omega)),
            \\
            &\trhoeps \nablax\Eeps\ceps \weak \overline{\rho\nablax c} \quad \text{in } L^\infty(0,T;L^{\frac{2\gamma}{\gamma+2}}(\Omega)),
            \quad 
            \Peff(\trhoeps)\weak \overline{\Peff} \quad \text{in } L^{\frac{5\gamma-3}{3\gamma}}(\OmegaT),
            \quad
            \\
            &\Eeps\ceps \weakstar c \quad \text{in } L^\infty(0,T;W^{1,2}(\Omega)),
            \quad
            \Eeps\ceps \to c \quad \text{in } C([0,T];L^{2}(\Omega)),
            \\
            &\tceps \to c \quad \text{in } C([0,T];L^2(\Omega)),
            \quad
            \partial_t \tilde{c}_\eps \weak \partial_t c \quad \text{in } L^2(\OmegaT).
        \end{aligned}
    \end{equation}
    Moreover, we have that the triplet $(\rho,\uvec,c)$ satisfies the regularity \eqref{wkSol reg} with $\Omegaeps$ replaced by $\Omega$, that
    \begin{align}
        &\int_0^T\int_\Omega 
        \rho\partial_t \varphi + \rho\uvec \cdot \nablax \varphi 
        \, \dd x \, \dd t
        =
        0,\label{hom3d limit cont}
        \\
        &\int_0^T \int_\Omega
        \paracoup c \partial_t \varphi - \kappa \nablax c \cdot \nablax \varphi - \coup(c-\rho)\varphi 
        \, \dd x \, \dd t = 0 \label{hom3d limit parab}
    \end{align}
    for any test function $\varphi \in \mathcal{D}(\RR^3_T)$,
    \begin{equation}\label{hom3d limit mom}
        \begin{aligned}
            &\int_0^T \int_\Omega 
            \rho\uvec \cdot \partial_t \varphivec + \rho\uvec \otimes\uvec : \nablax \varphivec + \overline{\Peff}\div\varphivec
            \, \dd x \, \dd t
            \\
            &\quad =
            \int_0^T \int_\Omega 
            \Svisc(\nablax\uvec):\nablax\uvec - \coup \overline{\rho\nablax c}\cdot \varphivec
            \, \dd x \, \dd t
        \end{aligned}
    \end{equation}
    for any test function $\varphivec \in \mathcal{D}(\OmegaT;\RR^3)$, and
    \begin{equation}\label{hom3d limit IC}
        \rho(0)=\rho_0,
        \qquad 
        (\rho\uvec)(0)=(\rho\uvec)_0,
        \qquad 
        c(0)=c_0
        \qquad \text{a.e.~in } \Omega.
    \end{equation}
\end{lemma}
\begin{proof}
    Due to Lemma~\ref{lem:Hom3d bds exts}, Lemma~\ref{lemPhi}, the Sobolev embedding $W^{1,2}(\Omega)\hookrightarrow L^6(\Omega)$, and the Banach--Alaoglu theorem, we have, after passing to a non-relabeled subsequence,
    \begin{equation}\label{Hom3d Alaoglu cv}
        \begin{aligned}
            &\trhoeps \weak \rho \quad \text{in } L^{\frac{5\gamma-3}{3\gamma}}(\OmegaT),
            \quad 
            \Peff(\trhoeps)\weak \overline{\Peff}\quad \text{in } L^{\frac{5\gamma-3}{3\gamma}}(\OmegaT),
            \quad 
            \tuveceps\weak \uvec \quad \text{in } L^2(0,T;W^{1,2}_0(\Omega)),
            \quad
            \\
            &\Phi^\mathrm{T}_\eps\trhoeps\tuveceps\otimes\tuveceps \weak \overline{\rho\uvec\otimes\uvec} \quad \text{in } L^2(0,T;L^{\frac{3\gamma}{\gamma+3}}(\Omega)),
            \quad 
            \Eeps\ceps \weakstar c \quad \text{in } L^\infty(0,T;W^{1,2}(\Omega)),
            \\
            &\trhoeps\nablax\Eeps\ceps \weakstar \overline{\rho\nablax c} \quad \text{in } L^\infty(0,T;L^{\frac{2\gamma}{\gamma+2}}(\Omega)),
            \quad
            \partial_t \tilde{c}_\eps \weak \partial_t c \quad \text{in } L^2(\OmegaT).
        \end{aligned}
    \end{equation}
    Note that the weak limits of the extensions $\Eeps\ceps$ and $\tilde{c}_\eps$ coincide thanks to the first estimate in \eqref{Hom3d bds exts II}. Furthermore, using the Sobolev embedding $W^{1,2}(\Omega)\hookrightarrow L^6(\Omega)$ and H\"older's inequality, we deduce from \eqref{Hom3d bds exts I}, \eqref{Hom3d cont exts}, \eqref{Hom3d mom exts}, \eqref{Hom3d est remainder G_eps}, and \Cref{lemPhi} that
    \begin{equation*}
        \|\partial_t \trhoeps \|_{L^s(0,T;W^{-1,s}(\Omega))}
        +
        \|\partial_t (\Phi_\eps^{\mathrm{T}}\trhoeps\tuveceps)\|_{L^s(0,T;W^{-1,s}(\Omega))}
        \leq 
        C
    \end{equation*}
    for some $s \in (1,\infty)$. Thanks to \eqref{Hom3D energybds}, \Cref{lemPhi}, and the Sobolev embedding $W^{1,2}(\Omega)\hookrightarrow L^6(\Omega)$, this implies, after passing to a non-relabeled subsequence,
    \begin{equation}\label{Hom3d Arzela cv}
        \trhoeps \to \rho \quad \text{in } \Cw([0,T];L^\gamma(\Omega)),
        \qquad 
        \Phi_\eps^{\mathrm{T}}\trhoeps \tuveceps \to \overline{\rho \uvec} \quad \text{in } \Cw([0,T];L^{\frac{2\gamma}{\gamma+1}}(\Omega)).
    \end{equation}
    Using the compact Sobolev embedding $L^\gamma(\Omega)\hookrightarrow\hookrightarrow W^{-1,2}(\Omega)$ and combining the first convergence in \eqref{Hom3d Arzela cv} with the third convergence in \eqref{Hom3d Alaoglu cv} yields
    \begin{equation}\label{Hom3d id lim momentum}
        \overline{\rho\uvec}= \rho\uvec \quad \text{a.e.~in } \OmegaT.
    \end{equation}
    Similarly, due to the compact Sobolev embedding $L^{\frac{2\gamma}{\gamma+1}}(\Omega)\hookrightarrow\hookrightarrow W^{-1,2}(\Omega)$, the second convergence in \eqref{Hom3d Arzela cv}, the third convergence in \eqref{Hom3d Alaoglu cv}, and \eqref{Hom3d id lim momentum}, we have
    \begin{equation}\label{Hom3d id lim conv term}
        \overline{\rho\uvec\otimes\uvec} 
        = 
        \rho\uvec\otimes\uvec \quad \text{a.e.~in } \OmegaT.
    \end{equation}
    By virtue of Lemma~\ref{lem:Hom3d bds exts} and since the embedding $W^{1,2}(\Omega)\hookrightarrow\hookrightarrow L^2(\Omega)$ is compact we conclude with the Aubin--Lions theorem that
    \begin{equation}\label{Hom3d strong conv c}
        \Eeps\ceps \to c \quad \text{in } C([0,T];L^{2}(\Omega)).
    \end{equation}
    Thanks to \eqref{Hom3d bds exts II} this implies 
    \begin{equation}\label{Hom3d strong conv c II}
        \tceps \to c \quad \text{in } C([0,T];L^2(\Omega)).
    \end{equation}
    In view of \eqref{Hom3d Arzela cv}--\eqref{Hom3d strong conv c II} we have shown \eqref{Hom3d weak limits}.
    Then we can pass to the limit $\eps\to 0$ in \eqref{Hom3d cont exts}, \eqref{Hom3d parab exts}, and \eqref{Hom3d mom exts}, and obtain that $(\rho,\uvec,c)$ satisfies \eqref{hom3d limit cont}--\eqref{hom3d limit mom}.
    Using the assumptions \eqref{thm:Hom IC} as well as Lemma~\ref{lemPhi}, we infer for the initial conditions that
    \begin{equation*}
        \begin{aligned}
            &\trhoeps(0)=\tilde{\rho}_{0\eps}\to \rho_0
            \quad \text{in } L^\gamma(\Omega),
            \qquad 
            (\Phi_\eps^{\mathrm{T}}\trhoeps\tuveceps) (0) = \Phi_\eps^{\mathrm{T}}\widetilde{(\rho\uvec)}_{0\eps} \to (\rho\uvec)_0 
            \quad \text{in } L^{\big(\frac{2\gamma}{\gamma+1}\big)^-}(\Omega),
            \\
            &(\Eeps\ceps)(0) = \Eeps c_{0\eps} \to c_0 \quad \text{in } W^{1,2}(\Omega).
        \end{aligned}
    \end{equation*}
    Thanks to the convergences in \eqref{Hom3d Arzela cv} and in \eqref{Hom3d strong conv c}, this leads to \eqref{hom3d limit IC}.
    In particular, we have that $c$ is the weak solution of the linear parabolic problem
    \begin{equation*}
        \begin{aligned}
            &\paracoup\partial_t c - \kappa \Deltax c + \coup c = \coup \rho 
            &&\text{in } \OmegaT,
            \\
            &\nablax c \cdot \mathbf{n}_{\partial\Omega} = 0
            &&\text{on } (0,T) \times \partial\Omega,
            \\
            &c(0)=c_0 &&\text{in } \Omega.
        \end{aligned}
    \end{equation*}
    Using parabolic regularity results (see e.g.~\cite[Theorem~11.29]{FeireislNovotny2017singlim}) yields $c \in C([0,T];W^{1,2}(\Omega))$. In total we have shown that the triplet $(\rho,\uvec,c)$ satisfies the regularity \eqref{wkSol reg} with $\Omegaeps$ replaced by $\Omega$.
    \end{proof}
    To complete the proof of \Cref{thm:Hom}, we have to identify
    \begin{equation*}
        \overline{\Peff} = \Peff(\rho),
        \quad 
        \overline{\rho\nablax c} = \rho\nablax c
        \quad \text{a.e.~in } \OmegaT.
    \end{equation*}
    These relations are valid once we show that $\trhoeps\to \rho$ strongly in $L^1(\OmegaT)$.
    To this end, we first deduce a weak compactness property for the effective viscous flux associated with the extended quantities.
    Then we exploit this property by using the technique from \cite{Feireisl2002} to show the strong convergence.
    \begin{lemma}\label{lem:Hom3d str cv density}
        Let the hypotheses and notations of \Cref{lem:Hom3d weak lim} hold true.
        Then we have 
        \begin{equation*}
            \trhoeps \to \rho \quad \text{in } L^{\big(\frac{5\gamma-3}{3}\big)^-}(\OmegaT) 
        \end{equation*}
        and, in particular,
        \begin{equation*}
            \overline{p_\coup} = p_\coup(\rho),
            \quad
            \overline{\rho\nablax c} = \rho\nablax c
            \qquad \text{a.e.~in } \OmegaT.
        \end{equation*}
    \end{lemma}
    \begin{proof}
    For any $k\in\NN$ we introduce the cut-off functions 
    \begin{equation*}
        T_k(r) := k T\Big(\frac{r}{k}\Big) \qquad \text{for } r \in [0,\infty),
    \end{equation*}
    for some fixed concave function $T\in C^\infty([0,\infty))$ satisfying 
    \begin{equation*}
        T(r)=r \quad \forall\, r \in [0,1],
        \qquad
        T(r)=2 \quad \forall \, r \in [3,\infty).
    \end{equation*}
    Due to \eqref{Hom3d cont exts}, we have
    \begin{equation*}
        \partial_t{T_k(\trhoeps)}
        +
        \div\big(T_k(\trhoeps) \tuveceps \big) 
        +
        \big(T_k^\prime(\trhoeps)\trhoeps - T_k(\trhoeps)\big)\div \tuveceps 
        =
        0
        \qquad \text{in } \mathcal{D}^\prime(\RR^3_T),
    \end{equation*}
    see e.g.~\cite[Lemma~11.13]{FeireislNovotny2017singlim}.
    This implies in view of \Cref{lem:Hom3d bds exts} that, up to a subsequence,
    \begin{equation*}%\label{Cw Tk}
        T_k(\trhoeps) \to \overline{T_k} \qquad \text{in } \Cw([0,T];L^{\infty^-}(\Omega)),
    \end{equation*}
    for some $\overline{T_k}\in \Cw([0,T];L^{\infty^-}(\Omega))$.
    Moreover, we conclude from \Cref{lem:Hom3d bds exts} by using the Banach--Alaoglu theorem that, up to a subsequence,
    \begin{equation*}%\label{wk Tk}
        T_k(\trhoeps) \weakstar \overline{T_k} \quad \text{in } L^\infty(0,T;L^\infty(\Omega)),
        \quad 
        \big(T_k^\prime(\trhoeps)\trhoeps - T_k(\trhoeps)\big) \div\tuveceps 
        \weak
        f 
        \quad \text{in } L^2(\OmegaT),
    \end{equation*}
    for some $f \in L^2(\OmegaT)$.
    With these relations, \Cref{lem:Hom3d mom exts}, \Cref{lem:Hom3d weak lim}, and \eqref{thm:Hom3D restr gamma,alpha}, we can apply \cite[Proposition~A.1]{OschmannWendt2026-exist} with
    \begin{equation*}
        q>6,
        \quad 
        r=\frac{5\gamma-3}{3\gamma},
        \quad 
        s=\frac{2\gamma}{\gamma+2},
        \quad
        z=\frac{2\gamma}{\gamma+1},
        \quad 
        \sigma = 2,
        \quad 
        w = \frac{5\gamma-3}{3}
    \end{equation*}
    to deduce
    \begin{equation}\label{EVF in Proof}
        \begin{aligned}
        \lim\limits_{\eps \to 0}
            &\int_0^T \int_\Omega 
            \varphi T_k(\trhoeps) \bigg( p_\coup(\trhoeps) - \Big( \bulkvisc + \frac{4}{3}\shearvisc\Big) \div\tuveceps\bigg)
            \, \dd x \, \dd t
            \\
            &\quad =
            \int_0^T \int_\Omega 
            \varphi \overline{T_k} \bigg( \overline{p_\coup} - \Big( \bulkvisc + \frac{4}{3}\shearvisc \Big)\div \uvec \bigg) 
            \, \dd x\, \dd t
            \qquad \forall\, \varphi \in \mathcal{D}(\OmegaT).
        \end{aligned}
    \end{equation}
    With \eqref{EVF in Proof} at hand we can use the techniques from \cite{Feireisl2002} to deduce the strong convergence of the density.
    The detailed arguments for this conclusion have been given in the proof of \cite[Lemma~5.3]{OschmannWendt2026-exist}. For the sake of brevity we thus omit them here.
    \end{proof}
    To complete the proof of \Cref{thm:Hom}, we use the strong and weak convergences derived in this section to perform the homogenization limit in the energy inequality.
    \begin{proof}[Proof of \Cref{thm:Hom} in the case $d=3$]
        In view of \Cref{lem:Hom3d weak lim} and \Cref{lem:Hom3d str cv density}, we only have to show that the triplet $(\rho,\uvec,c)$ satisfies the energy inequality \eqref{wkSol energy} on the unperforated domain $\OmegaT$.
        To do so, we fix $\psi \in C^\infty_c([0,T))$ with $\psi\geq 0$.
        Then $(\rhoeps,\uveceps,\ceps)$ satisfies
        \begin{equation*}
            \begin{aligned}
                &-\int_0^T \partial_t \psi \int_{\Omegaeps}
                \frac{1}{2}\rhoeps|\uveceps|^2
                +
                W(\rhoeps)
                +
                \frac{\coup}{2}|\rhoeps-\ceps|^2  
                +
                \frac{\kappa}{2}|\nablax \ceps|^2
                \, \dd x \, \dd t
                \\
                &\quad +
                \int_0^T \psi \int_{\Omegaeps}
                \Svisc(\nablax\uveceps):\nablax\uveceps + \paracoup|\partial_t \ceps|^2
                \, \dd x \, \dd t
                \\
                &\leq 
                \psi(0)\int_{\Omegaeps}
                \frac{|(\rho\uvec)_{0\eps}|^2}{2\rho_{0\eps}}
                +
                W(\rho_{0\eps})
                +
                \frac{\coup}{2}|\rho_{0\eps}-c_{0\eps}|^2 
                +
                \frac{\kappa}{2}|\nablax c_{0\eps}|^2
                \, \dd x,
            \end{aligned}
        \end{equation*}
        which we can rewrite as
        \begin{equation}\label{Hom3d final energy rewritten}
        \begin{aligned}
            &-\int_0^T \partial_t \psi \int_\Omega 
            \frac{1}{2}\trhoeps|\tuveceps|^2
            + 
            W(\trhoeps)
            +
            \frac{\coup}{2}|\trhoeps-\tceps|^2 
            \, \dd x \, \dd t
            -
            \int_0^T \partial_t \psi \int_{\Omegaeps}
            \frac{\kappa}{2}|\nablax \ceps|^2 
            \, \dd x \, \dd t
            \\
            &\quad 
            +
            \int_0^T \psi\int_\Omega 
            \Svisc(\nablax\tuveceps):\nablax\tuveceps + \paracoup|\partial_t \tilde{c}_\eps|^2
            \, \dd x \, \dd t
            \\
            &\leq 
            \psi(0) \int_{\Omega}
            \frac{|\widetilde{(\rho\uvec)}_{0\eps}|^2}{2\tilde{\rho}_{0\eps}}
            +
            W(\tilde{\rho}_{0\eps})
            +
            \frac{\coup}{2} |\tilde{\rho}_{0\eps}-\tilde{c}_{0\eps}|^2
            +
            \frac{\kappa}{2}|\nablax\Eeps c_{0\eps}|^2
            \, \dd x 
            \\
            &\quad 
            -
            \psi(0)
            \int_{\Omega\setminus\Omegaeps}
            \frac{\kappa}{2}|\nablax \Eeps c_{0 \eps}|^2 
            \, \dd x.
        \end{aligned}
        \end{equation}
        By virtue of the convergences in \Cref{lem:Hom3d weak lim} and \Cref{lem:Hom3d str cv density}, we have 
        \begin{equation}\label{Hom3d final limit I}
        \begin{aligned}
            &\lim\limits_{\eps\to 0}
            \int_0^T \partial_t \psi \int_\Omega 
            \frac{1}{2}\trhoeps|\tuveceps|^2
            + 
            W(\trhoeps)
            +
            \frac{\coup}{2}|\trhoeps-\tceps|^2
            \, \dd x \, \dd t
            \\
            &=
            \int_0^T \partial_t \psi \int_\Omega 
            \frac{1}{2}
            \rho |\uvec|^2
            +
            W(\rho)
            +
            \frac{\coup}{2}|\rho-c|^2 
            \, \dd x \, \dd t,
        \end{aligned}
        \end{equation}
        and moreover, due to the weak lower semi-continuity of convex functionals,
        \begin{equation}\label{Hom3d final limit II}
        \begin{aligned}
            &\int_0^T \psi \int_\Omega 
            \Svisc(\nablax\uvec):\nablax \uvec + \paracoup|\partial_t c|^2
            \, \dd x \, \dd t
            \\
            &\leq 
            \liminf\limits_{\eps\to0}
            \int_0^T \psi \int_\Omega 
            \Svisc(\nablax\tuveceps):\nablax\tuveceps + \paracoup|\partial_t \tilde{c}_\eps|^2
            \, \dd x \, \dd t.
        \end{aligned}
        \end{equation}
        In view of the assumptions \eqref{thm:Hom IC} on the initial data, we have
        \begin{equation}\label{Hom3d final limit III}
        \begin{aligned}
            &\lim\limits_{\eps \to 0}
            \int_\Omega 
            \frac{|\widetilde{(\rho\uvec)}_{0\eps}|^2}{2\tilde{\rho}_{0\eps}}
            +
            W(\tilde \rho_{0\eps})
            +
            \frac{\coup}{2}|\tilde{\rho}_{0\eps}-\tilde{c}_{0\eps}|^2
            +
            \frac{\kappa}{2}|\nablax \Eeps c_{0\eps}|^2
            \, \dd x 
            \\
            &=
            \int_\Omega 
            \frac{|(\rho\uvec)_0|^2}{2\rho_0}
            +
            W(\rho_0)
            +
            \frac{\coup}{2}|\rho_0-c_0|^2
            +
            \frac{\kappa}{2}|\nablax c_0|^2
            \, \dd x,
        \end{aligned}
        \end{equation}
        and moreover, thanks to \eqref{thm:Hom IC} and $|\Omega\setminus\Omegaeps|\xrightarrow{\eps\to0} 0$,
        \begin{equation}\label{Hom3d final limit IV}
        \begin{aligned}
            \lim\limits_{\eps\to 0}
            \int_{\Omega\setminus\Omegaeps}
            \frac{\kappa}{2} |\nablax \Eeps c_{0\eps}|^2
            \, \dd x 
            = 
            0.
        \end{aligned}
        \end{equation}
        In order to pass to the limit $\eps \to 0$ in the second term on the left-hand side in \eqref{Hom3d final energy rewritten}, we have to exploit the parabolic equation which is satisfied by $(\rho,\uvec,c)$ and $(\rhoeps,\uveceps,\ceps)$ on the unperforated and perforated domain, respectively.
        By a density argument, we may choose $\varphi = c$ as a test function in \eqref{hom3d limit parab} and obtain
        \begin{equation}\label{Hom3d final parab tested}
            \int_0^T \int_\Omega 
            \kappa |\nablax c|^2 
            \, \dd x \, \dd t 
            =
            \int_0^T \int_\Omega 
            \paracoup c \partial_t c 
            -
            \coup |c|^2
            +
            \coup  c \rho
            \, \dd x \, \dd t.
        \end{equation}
        Similarly, we may choose $\varphi = \ceps$ as a test function in \eqref{wkSol parab} and obtain
        \begin{equation*}
        \begin{aligned}
            \int_0^T \int_{\Omegaeps} 
            \kappa |\nablax \ceps|^2 
            \, \dd x \, \dd t
            &=
            \int_0^T \int_{\Omegaeps}
            \paracoup\ceps \partial_t \ceps
            -
            \coup|\ceps|^2 
            +
            \coup\ceps\rhoeps
            \, \dd x \, \dd t
            \\
            &=
            \int_0^T \int_\Omega 
            \paracoup\tceps \partial_t \tilde{c}_\eps 
            -
            \coup |\tceps|^2 
            +
            \coup \tceps \trhoeps
            \, \dd x \, \dd t .
        \end{aligned}
        \end{equation*}
        With \eqref{Hom3d weak limits} and \eqref{Hom3d final parab tested}, passing to the limit $\eps \to 0$ leads to
        \begin{equation*}
            \lim\limits_{\eps \to 0}
            \int_0^T \int_{\Omegaeps}
            \kappa |\nablax \ceps|^2 
            \, \dd x \, \dd t
            =
            \int_0^T \int_\Omega 
            \paracoup c \partial_t c 
            -
            \coup|c|^2 
            +
            \coup c \rho
            \, \dd x \, \dd t
            =
            \int_0^T \int_\Omega 
            \kappa |\nablax c|^2 
            \, \dd x \, \dd t
        \end{equation*}
        and thus, in particular,
        \begin{equation}\label{Hom3d final norm grad c limit}
            \lim\limits_{\eps \to 0}
            \int_0^T \partial_t \psi \int_{\Omegaeps}
            \frac{\kappa}{2}
            |\nablax \ceps|^2 
            \, \dd x \, \dd t
            =
            \int_0^T \partial_t \psi \int_\Omega 
            \frac{\kappa}{2} |\nablax c|^2 
            \, \dd x \, \dd t.
        \end{equation}
        Combining \eqref{Hom3d final energy rewritten}--\eqref{Hom3d final norm grad c limit} leads to
        \begin{equation*}
        \begin{aligned}
            &-\int_0^T \partial_t \psi \int_\Omega 
            \frac{|\rho\uvec|^2}{2\rho}
            +
            W(\rho)
            +
            \frac{\coup}{2}|\rho-c|^2
            +
            \frac{\kappa}{2}|\nablax c|^2
            \, \dd x \, \dd t
            \\
            &\quad 
            +
            \int_0^T\psi\int_\Omega 
            \Svisc(\nablax\uvec):\nablax\uvec + \paracoup|\partial_t c|^2
            \, \dd x \, \dd t
            \\
            &\leq 
            \psi(0) \int_\Omega 
            \frac{|(\rho\uvec)_0|^2}{2\rho_0}
            +
            W(\rho_0)
            +
            \frac{\coup}{2}|\rho_0-c_0|^2
            +
            \frac{\kappa}{2}|\nablax c_0|^2
            \, \dd x,
        \end{aligned}
        \end{equation*}
        which is precisely the energy inequality \eqref{wkSol energy} for $(\rho,\uvec,c)$ on the unperforated domain $\OmegaT$.
    \end{proof}

\section{Homogenization in 2D}\label{sec:Hom2D}
The core of this section is the proof of Theorem~\ref{thm:Hom} in the case $d=2$.
Since the calculations are mostly done similarly, we will just focus on the differences to the 3D case.
Our main reference here will be the article \cite{NecasovaOschmann2023}, where homogenization of the 2D compressible evolutionary NSE has been investigated.
We will therefore use several technical tools proven in there as a black-box and give references to them whenever necessary.
In the 2D setting, the domain $\Omega$ is, for $\eps>0$, perforated by holes of the size
\begin{equation*}
    a_\eps = \exp(-\eps^{-\alpha}), 
    \qquad 
    \alpha>2,
\end{equation*}
see \eqref{def:domain}.
This scaling is connected to the fact that the harmonic (also called Newtonian) capacity of the holes in 2D is rather logarithmic than polynomial. Note also that this scaling is not unique: in fact, we can also consider holes of size $a_\eps = \eps^s \exp(-\eps^{-\alpha})$ for any $s \in \RR$; see \cite[Section~2]{Allaire1990a} for details. Finally, we will assume that the adiabatic exponent satisfies $\gamma > 2$, which is connected to the space-time dimension of the underlying domain, see \cite[Section~7]{OschmannPokorny2023} for an explanation on this.
\subsection{Uniform estimates}\label{subsec:UnifEsts2D}
Let us assume that the hypotheses and notations of Theorem~\ref{thm:Hom} with $d=2$ hold true.
By the same arguments leading to \eqref{Hom3D energybds} we derive from the energy inequality \eqref{wkSol energy} the uniform bound
\begin{equation}\label{Hom2D energybds I}
        \begin{aligned}
            &\|\sqrt{\rhoeps}\uveceps\|_{L^\infty(0,T;L^2(\Omegaeps))}
            +
            \|\rhoeps\|_{L^\infty(0,T;L^\gamma(\Omegaeps))}
            +\|c_\eps\|_{L^\infty(0,T;W^{1,2}(\Omegaeps))}
            \\
            &\quad 
            +
            \|\uveceps\|_{L^2(0,T;W^{1,2}(\Omegaeps))}
            +
            \|\partial_t\ceps\|_{L^2(0,T;L^2(\Omegaeps))}
            \leq C.
        \end{aligned}
    \end{equation}
To obtain a better uniform bound for the density $\rho_\eps$, we proceed as in the 3D case (see \Cref{sec:Hom3D}) and use a suitable Bogovski\u{\i} type operator with argument $\rho_\eps^\theta$ for some $\theta > 0$. In 2D, it turns out that such an operator needs to be constructed in a fine manner, and an additional parameter $\varpi_\eps \geq \eps^\frac{\alpha}{2} a_\eps^{-1}$ needs to be introduced to capture the interplay between the holes' distance $\eps$ and their exponentially small size $a_\eps$. 
Such an operator has been constructed in \cite[Appendix~A]{NecasovaOschmann2023}, the properties of which we state here as a lemma:
\begin{lemma}\label{lem:Bogeps 2D}
    For $\eps\in(0,1)$, let $\Omegaeps$ be defined as in \eqref{def:domain} with $d=2$, and let $s \in (1,\infty)$.
    Then there exists a bounded linear operator $\Bogeps\colon L^s_0(\Omegaeps) \to W^{1,s}_0(\Omegaeps;\RR^2)$ such that for any $f \in L_0^s(\Omega_\eps)$, we have
    \begin{align}\label{lem:Bogeps 2D non-ext}
        \div \Bogeps(f) = f \text{ a.e.~in } \Omega_\eps, \qquad \|\Bogeps(f)\|_{W^{1,s}_0(\Omegaeps)}^s \leq C(1+C_{\eps,s})\|f\|_{L^s(\Omegaeps)}^s,
    \end{align}
    where
    \begin{align*}
        C_{\eps,s} := \begin{cases}
            \eps^{-\alpha} a_\eps^{2-s} \big|\log(\eps^{\frac{\alpha}{2}}a_\eps^{-1})\big|^{-s} \big|\eps^{\frac{\alpha(2-s)}{2}}a_\eps^{s-2} - 1 \big| & \text{if } s \neq 2,\\
            \eps^{-\alpha} \big|\log(\eps^{\frac{\alpha}{2}}a_\eps^{-1})\big|^{-1} & \text{if } s=2.
        \end{cases}
    \end{align*}
    Moreover, for any $s\in (1,2]$, this operator can be extended to a linear operator
    \begin{align}\label{lem:Bogeps 2D ext}
    \mathcal{B}_\eps\colon \big[\dot{W}^{1,s^\prime}(\Omegaeps)\big]^\ast :=  \big\{ f \in \big[W^{1, s'}(\Omega_\eps)\big]^\ast : \langle f, 1\rangle=0 \big\} \to L^s(\Omega_\eps;\RR^3),
    \end{align}
    such that, for any $f\in \big[\dot{W}^{1,s^\prime}(\Omegaeps)\big]^\ast$,
    \begin{align*}
    -\int_{\Omegaeps}\Bogeps(f)\cdot \nablax \phi\, \dd x
    = \langle f,  \phi\rangle \quad \forall\,\phi\in W^{1,s'}(\Omega_\eps), 
    \quad
    \|\mathcal{B}_\eps (f) \|_{L^s(\Omega_\eps)} \leq C \|f\|_{[W^{1,s^\prime}(\Omega_\eps)]^\ast}.
    \end{align*}
\end{lemma}
\begin{remark}
    For any $s \in (1,2]$, we have $C_{\eps,s}\leq C$ and thus,
    \begin{equation*}
        \Bogeps \colon L^s_0(\Omegaeps) \to W^{1,s}_0(\Omegaeps;\RR^2)
    \end{equation*}
    is uniformly bounded with respect to $\eps$.
\end{remark}
\begin{remark}
    The construction of the operator $\Bogeps\colon L^s_0(\Omegaeps) \to W^{1,s}_0(\Omegaeps;\RR^2)$ satisfying \eqref{lem:Bogeps 2D non-ext} has been shown in \cite[Appendix~A]{NecasovaOschmann2023}.
    The extension of the operator $\Bogeps$ in \eqref{lem:Bogeps 2D ext} is non-trivial and has been used for the analysis in \cite{NecasovaOschmann2023}, however, a proof for this extension was not given there. Thus, for the sake of completeness, we give a corresponding proof in the Appendix~\ref{App:Bog2D}.
\end{remark}
Having the operator $\Bogeps$ at hand, we obtain an improved estimate for the density analogous to \Cref{lem:impr prs Hom3D} in the 3D case.
\begin{lemma}\label{lem:impr prs Hom2D}
    Let the hypotheses and notations of \Cref{thm:Hom} with $d=2$ hold true.
    Then we have 
    \begin{equation}\label{impr prs Hom2D}
        \|\rhoeps\|_{L^{(2\gamma-1)^-}(\OmegaT)} + \|\Peff(\rhoeps)\|_{L^{\frac{(2\gamma-1)^-}{\gamma}}(\OmegaT)} 
        \leq 
        C.
    \end{equation}
\end{lemma}
\begin{proof}
    We fix $\theta \in (0,\gamma-1)$ and $\psi \in C^\infty_c((0,T))$ with $\psi \geq 0$.
    As in the proof of \Cref{impr prs Hom3D} we may assume without loss of generality that
    \begin{equation*}
        \rhoeps^\theta,\,\,\partial_t(\rhoeps^\theta) \in L^\infty(0,T;L^\infty(\Omegaeps)),
    \end{equation*}
    and set
    \begin{equation*}
        \varphivec_\eps := \Bogeps \left(\rhoeps^\theta - \fint_{\Omegaeps} \rhoeps^\theta \right),
    \end{equation*}
    where $\Bogeps$ denotes the Bogovski\u{\i} operator from \Cref{lem:Bogeps 2D}.
    Proceeding as in the proof of \Cref{lem:impr prs Hom3D}, we verify \eqref{norm trace}, \eqref{deriv test fct Bogeps}, and
    \begin{equation*}
        \begin{aligned}
            \int_0^T \psi \int_{\Omegaeps} \Peff(\rhoeps)\rhoeps^\theta \, \dd x \, \dd t
            =:
            \sum\limits_{i=1}^6 I_i^\eps,
        \end{aligned}
    \end{equation*}
    where $I_i^\eps$ are defined as in the proof of \Cref{lem:impr prs Hom3D}.
    By virtue of \eqref{thm:Hom2D restr gamma,alpha}, \eqref{Hom2D energybds I}, and \Cref{lem:Bogeps 2D}, we have
    \begin{equation*}
        \begin{aligned}
            \sum\limits_{i=1}^5 |I_i^\eps| 
            \leq 
            C(1 + \|\psi\|_{L^\infty((0,T))}
            +
            \|\partial_t \psi\|_{L^1((0,T))}
            )
            +
            \frac{3}{4} \int_0^T \psi \int_{\Omegaeps}
            \rhoeps^{\gamma + \theta}
            \, \dd x \, \dd t.
        \end{aligned}
    \end{equation*}
    A detailed exposition of the corresponding estimates can be found in \cite[Appendix B]{Bravin2024}, therefore we omit the details here.
    For the last term $I_6^\eps$, we estimate by the same token using the Sobolev embedding $W^{1,\frac{\gamma}{\gamma-1}}(\Omega)\hookrightarrow L^{\frac{2\gamma}{\gamma-2}}(\Omega)$, which holds thanks to $\gamma>2$ and $d=2$,
    \begin{equation*}
        \begin{aligned}
            |I_6^\eps|
            &\leq 
            C\|\psi\|_{L^\infty((0,T))}\|\rhoeps\|_{L^\infty(0,T;L^{\gamma}(\Omegaeps))} \|\nablax\ceps\|_{L^\infty(0,T;L^{2}(\Omegaeps))} \|\varphivec_\eps\|_{L^\infty(0,T;L^{\frac{2\gamma}{\gamma-2}}(\Omegaeps))}
            \\
            &\leq 
            C\|\psi\|_{L^\infty((0,T))}\|\varphivec_\eps\|_{L^\infty(0,T;W^{1,\frac{\gamma}{\gamma-1}}(\Omegaeps))}
            %\\
            %&
            \leq 
            C\|\psi\|_{L^\infty((0,T))} \|\rhoeps \|_{L^\infty(0,T;L^{\frac{\theta\gamma}{\gamma-1}}(\Omegaeps))}^\theta
            \\
            &\leq 
            C\|\psi\|_{L^\infty((0,T))}.
        \end{aligned}
    \end{equation*}
    In the third inequality we have used that
    \begin{equation*}
        \frac{\gamma}{\gamma-1}<2,
    \end{equation*}
    which holds thanks to $\gamma >2$, and in the last line we have used the fact $\theta \in (0,\gamma-1)$, which implies
    \begin{equation*}
        \frac{\theta\gamma}{\gamma-1} < \gamma.
    \end{equation*}
    Using that \eqref{dec prs reg}--\eqref{dec prs grwth}, we arrive at
    \begin{equation*}
        \int_0^T \psi \int_{\Omegaeps} \rhoeps^{\gamma+\theta} \, \dd x \, \dd t 
        \leq 
        C\big(1 + \|\psi\|_{L^\infty((0,T))} + \|\partial_t \psi\|_{L^1((0,T))}\big).
    \end{equation*}
    From this relation, we conclude \eqref{impr prs Hom2D} by using the same arguments as in the proof of \Cref{lem:impr prs Hom3D}.
\end{proof}
\subsection{Extensions of functions}\label{subsec:ExtOfFcts2D}
Analogously to the 3D case, we extend a weak solution $(\rhoeps,\uveceps,\ceps)$ defined on the porous domain $(0,T)\times\Omegaeps$ to the unperforated domain $\OmegaT$ by using the zero extension for the density $\rhoeps$ and the velocity $\uveceps$, and the extension operator $\Eeps$ from \Cref{lem:extEeps} for the relaxation parameter $\ceps$.
The uniform bounds in \eqref{Hom2D energybds I} and \Cref{lem:impr prs Hom2D} imply the following uniform bounds for the extended quantities $\trhoeps,\tuveceps,\tceps,\Eeps\ceps$.
\begin{lemma}\label{lem:Hom2d bds exts}
    Let the hypotheses and notations of \Cref{thm:Hom} with $d=2$ hold true.
    Then we have 
    \begin{equation}\label{Hom2d bds exts I}
        \begin{aligned}
            &\|\sqrt{\trhoeps}\tuveceps\|_{L^\infty(0,T;L^2(\Omega))}
            +
            \|\trhoeps\|_{L^\infty(0,T;L^\gamma(\Omega))}
            +
            \|\Eeps\ceps\|_{L^\infty(0,T;W^{1,2}(\Omega))}
            +
            \|\tceps\|_{L^\infty(0,T;L^2(\Omega))}
            \\
            &\quad +
            \|\tuveceps\|_{L^2(0,T;W^{1,2}(\Omega))}
            +
            \|\partial_t \Eeps\ceps\|_{L^2(\OmegaT)}
            +
            \|\partial_t\tceps\|_{L^2(\OmegaT)}
            +
            \|\trhoeps\|_{L^{(2\gamma-1)^-}(\OmegaT)}
            \\
            &\quad +
            \|\Peff(\trhoeps)\|_{L^{\frac{(2\gamma-1)^-}{\gamma}}(\OmegaT)}
            \leq 
            C.
        \end{aligned}
    \end{equation}
    Moreover, we have
    \begin{equation}\label{Hom2d bds exts II}
        \begin{aligned}
            &\|\Eeps\ceps-\tceps\|_{L^\infty(0,T;L^2(\Omega))}
            =
            \|\Eeps\ceps\|_{L^\infty(0,T;L^{2}(\Omega\setminus \Omegaeps))}
            \leq 
            C\eps^{-\frac{1}{2}}\exp\Big(-\frac{1}{4\eps^\alpha}\Big),
            \\
            &\|\nablax\Eeps\ceps\|_{L^\infty(0,T;L^1(\Omega\setminus\Omegaeps))}
            \leq 
            C \eps^{-1}\exp\Big(-\frac{1}{2\eps^{\alpha}} \Big).
        \end{aligned}
    \end{equation}
\end{lemma}
\begin{proof}
    The estimate \eqref{Hom2d bds exts I} follows from \eqref{Hom2D energybds I} and \Cref{lem:impr prs Hom2D} by using the properties of the extension operator in \Cref{lem:extEeps}.
    With \eqref{Hom2D energybds I} and \eqref{porous dom number density}, we estimate further by using the Sobolev embedding $W^{1,2}(\Omega) \hookrightarrow L^4(\Omega)$
    \begin{equation*}
    \begin{aligned}
        &\|\Eeps\ceps\|_{L^\infty(0,T;L^2(\Omega\setminus\Omegaeps))}
        \leq 
        C\|\Eeps\ceps\|_{L^\infty(0,T;L^4(\Omega))}|\Omega\setminus\Omegaeps|^{^{\frac{1}{4}}}
        \leq 
        C\eps^{-\frac{1}{2}}\exp\Big(-\frac{1}{4\eps^{\alpha}}\Big),
        \\
        &\|\nablax\Eeps\ceps\|_{L^\infty(0,T;L^1(\Omega\setminus\Omegaeps))}
        \leq 
        C\|\nablax\Eeps\ceps\|_{L^\infty(0,T;L^2(\Omega))}|\Omega\setminus\Omegaeps|^{\frac{1}{2}} 
        \leq 
        C \eps^{-1} \exp \Big(-\frac{1}{2\eps^\alpha} \Big).
    \end{aligned}
    \end{equation*}
\end{proof}
As in the 3D case, we immediately verify that the extended quantities $\trhoeps,\tuveceps,\tceps,\Eeps\ceps$ fulfill
 \begin{align}
            &\int_0^T \int_\Omega 
            \trhoeps \partial_t\varphi + \trhoeps\tuveceps \cdot \nablax \varphi 
            \, \dd x \, \dd t = 0,\label{Hom2d cont exts}
            \\
            &\int_0^T \int_\Omega 
            \paracoup\tilde{c}_\eps \partial_t \varphi - \kappa \nablax \Eeps\ceps \cdot \nablax \varphi 
            -
            \coup (\tilde{c}_{\eps} - \trhoeps)\varphi 
            \, \dd x \, \dd t
            = \langle \mathcal{F}_{\eps},  \varphi\rangle,\label{Hom2d parab exts}
    \end{align}
    for any $\varphi \in \mathcal{D}(\RR^3_T)$, where
    \begin{equation*}
        \langle\mathcal{F}_\eps,\varphi\rangle
        :=
        -
        \int_0^T \int_{\Omega\setminus\Omegaeps} 
        \kappa \nablax \Eeps\ceps \cdot \nablax \varphi 
        \, \dd x \, \dd t
        \xrightarrow{\eps \to 0} 0.
    \end{equation*}
Again, as in the 3D framework, we have to construct appropriate cut-off functions $\Phi_\eps$ in order to derive a momentum equation for the extended quantities on the unperforated domain.
Such suitable cut-off functions have been constructed in \cite{Bravin2024,NecasovaOschmann2023}.
More precisely, we have:
\begin{lemma}[see {\cite[Lemma~4.1]{NecasovaOschmann2023}}]\label{lem:Phi2D}
Let $\Omega\subseteq \RR^2$ be a bounded domain with smooth boundary.
For $\eps \in (0,1)$, let $\Omegaeps$ be defined by \eqref{def:domain} with $d=2$.
Then there is a solenoidal matrix-valued function $\Phi_\eps$ such that
\begin{align*}
\Phi_\eps \in W^{1,\infty^-}(\Omega; \RR^{2\times 2})\cap L^\infty(\Omega; \RR^{2\times 2}),
\quad
\Phi_\eps =0 \text{ on } \Omega\setminus \Omega_\eps,
\quad 
\Phi_\eps =\mathbb{I} \text{ on } \Omega\setminus \bigcup_{i \in K_\eps} B_{\eps^{\frac{\alpha}{2}}}(\eps x_i).
\end{align*}
Moreover, $\|\Phi_\eps\|_{L^\infty(\Omega)}\leq C$, and for any $1\leq s< \infty$,
\begin{align*}
\|\Phi_\eps-\mathbb{I}\|_{L^s(\Omega)} &\leq C \eps^{\frac{\alpha-2}{s}}
&&\|\nabla\Phi_\eps\|_{L^2(\Omega)} \leq C \eps^{-1} \big|\log (\eps^{\frac{\alpha}{2}}a_\eps^{-1})\big|^{-\frac12}.
\end{align*}
In turn, for any $\varphivec \in C_c^\infty(\Omega;\RR^2)$ and any $s \leq 2$,
\begin{align*}
\|\nabla (\Phi_\eps \varphivec)-\Phi_\eps \nabla\varphivec\|_{L^s(\Omega)} &\leq C \eps^{-1} \big|\log (\eps^{\frac{\alpha}{2}}a_\eps^{-1})\big|^{-\frac12} \|\varphivec\|_{L^\frac{2s}{2-s}(\Omega)}
\end{align*}
with the convention $1/0 := \infty$.
\end{lemma}
With the cut-off function $\Phi_\eps$ from \Cref{lem:Phi2D}, we derive the following version of the momentum equation for the extended quantities $\trhoeps,\tuveceps,\tceps,\Eeps\ceps$ on the unperforated domain $\OmegaT$:
\begin{lemma}\label{lem:Hom2d mom exts}
    Let the hypotheses and notations of \Cref{thm:Hom} with $d=2$ hold true and let $\Phi_\eps$ denote the cut-off function from \Cref{lem:Phi2D}.
    Then we have for any $\varphivec\in \mathcal{D}(\OmegaT;\RR^2)$ that
    \begin{equation}\label{Hom2d mom exts}
        \begin{aligned}
            &\int_0^T \int_\Omega 
            %\left(
            \Phi_\eps^\mathrm{T} \trhoeps\tuveceps \cdot \partial_t \varphivec 
            + 
            \Phi_\eps^\mathrm{T}\trhoeps\tuveceps \otimes \tuveceps : \nablax \varphivec
            +
            \Peff(\trhoeps) \div\varphivec
            %\right)
            \, \dd x \, \dd t
            \\
            &\quad -
            \int_0^T \int_\Omega 
            %\left(
            \Svisc(\nablax\tuveceps):\nablax\varphivec
            -
            \coup \trhoeps \nablax \Eeps \ceps \cdot \varphivec 
            %\right)
            \, \dd x \, \dd t
            =
            \langle \mathcal{\mathcal{G}_\eps,\varphivec \rangle },
        \end{aligned}
    \end{equation}
    where
    \begin{equation*}
        \begin{aligned}
            \langle \mathcal{G_\eps,\varphivec\rangle }
            &:=
            \int_0^T \int_\Omega 
            %\Big(
            \trhoeps\tuveceps\otimes \tuveceps : \big(\Phi_\eps\nablax\varphivec - \nablax (\Phi_\eps\varphivec)\big)
            +
            \Peff(\trhoeps)(\mathbb{I}-\Phi_\eps) : \nablax \varphivec
            %\Big)            
            \, \dd x \, \dd t 
            \\
            &\quad +
            \int_0^T \int_\Omega \Svisc(\nablax\tuveceps) : \Big( \big( \Phi_\eps-\mathbb{I}\big) \nablax \varphivec + \big( \nablax (\Phi_\eps\varphivec) - \Phi_\eps \nablax \varphivec\big)\Big)
            \, \dd x \, \dd t
            \\
            &\quad +
            \int_0^T \int_\Omega 
            \coup \trhoeps \nablax \Eeps\ceps \cdot\big( \mathbb{I} - \Phi_\eps \big) \cdot \varphivec
            \, \dd x \, \dd t.
        \end{aligned}
    \end{equation*}
    Furthermore, we find for any $r \in \left(\frac{2\gamma-1}{\gamma-1},\infty\right)$ some $\delta \in (0,\infty)$ such that
    \begin{equation}\label{Hom2d est remainder}
        |\langle \mathcal{G}_\eps,\varphivec\rangle | \leq C \eps^\delta \|\varphivec\|_{L^r(0,T;W^{1,r}(\Omega))}
        \quad \forall\, \varphivec \in \mathcal{D}(\OmegaT;\RR^2).
    \end{equation}
    In particular, we have 
    \begin{equation}\label{Hom2d convergence remainder}
        \mathcal{G}_\eps \to 0 \quad \text{in } L^{\big(\frac{2\gamma-1}{\gamma}\big)^-}(0,T;W^{-1,\big(\frac{2\gamma-1}{\gamma}\big)^-}(\Omega;\RR^2)).
    \end{equation}
\end{lemma}
\begin{proof}
    Equation \eqref{Hom2d mom exts} follows by the same lines as in the proof of \Cref{lem:Hom3d mom exts}.
    To see \eqref{Hom2d est remainder}, we fix $r \in \left(\frac{2\gamma-1}{\gamma-1},\infty\right)$ and $\varphivec \in C^\infty_c(\OmegaT;\RR^2)$.
    We first notice that 
    \begin{equation*}
        \frac{2\gamma-1}{\gamma-1}>2,
        \qquad 
        \frac{2\gamma-1}{\gamma} <2,
        \qquad 
        \frac{2\gamma}{\gamma+2}>1
        \qquad \forall\, \gamma \in (2,\infty).
    \end{equation*}
    In view of the first assumption in \eqref{thm:Hom2D restr gamma,alpha}, we thus find some $\hat{r} \in \left(1,\frac{2\gamma-1}{\gamma}\right)$ and some $\sigma \in (1,\infty)$ large enough such that
    \begin{equation*}
        \frac{1}{\hat{r}} + \frac{1}{r} + \frac{1}{\sigma} <1,
        \qquad 
        \frac{\gamma+2}{2\gamma} + \frac{1}{\sigma} <1.
    \end{equation*}
    Thanks to \Cref{lem:Hom2d bds exts}, \Cref{lem:Phi2D}, and the Sobolev embedding $W^{1,r}(\Omega) \hookrightarrow L^\infty(\Omega)$, which holds due to $r>2$, we have
    \begin{equation*}
        \begin{aligned}
            |\langle \mathcal{G}_\eps,\varphivec\rangle|
            &\leq 
           C\big(\|\trhoeps\tuveceps\otimes\tuveceps\|_{L^2(\OmegaT)} + \|\Svisc(\nablax\tuveceps)\|_{L^2(\OmegaT)} \big) \| \Phi_\eps\nablax\varphivec - \nablax (\Phi_\eps \varphivec)\|_{L^2(\OmegaT)}
           \\
           &\quad 
           +
           C\big( \|\Peff(\trhoeps)\|_{L^{\hat{r}}(\OmegaT)} + \|\Svisc(\nablax\tuveceps)\|_{L^{\hat{r}}(\OmegaT)} \big)\|\mathbb{I} - \Phi_\eps\|_{L^\sigma(\Omega)} \|\nablax \varphivec\|_{L^r(\OmegaT)}
           \\
           &\quad 
           +
           C\|\trhoeps\nablax\Eeps\ceps\|_{L^\infty(0,T;L^{\frac{2\gamma}{\gamma+2}}(\Omega))} \|\mathbb{I}-\Phi_\eps\|_{L^\sigma(\Omega)}\|\varphivec\|_{L^1(0,T;L^{\infty}(\Omega))}
           \\
           &\leq 
           C\big( \eps^{\frac{\alpha-2}{2}}\|\varphivec\|_{L^2(0,T;L^\infty(\Omega))} + \eps^{\frac{\alpha-2}{\sigma}} \|\nablax \varphivec\|_{L^r(\OmegaT)} + \eps^{\frac{\alpha-2}{\sigma}}\|\varphivec\|_{L^1(0,T;L^\infty(\Omega))} \big)
           \\
           &\leq 
           C\eps^\delta \|\varphivec\|_{L^r(0,T;W^{1,r}(\Omega))}
        \end{aligned}
    \end{equation*}
    with 
    \begin{equation*}
        \delta := \min\left\{\frac{\alpha-2}{2},\frac{\alpha-2}{\sigma} \right\}>0
    \end{equation*}
    due to the second assumption in \eqref{thm:Hom2D restr gamma,alpha}.
    The convergence \eqref{Hom2d convergence remainder} follows from \eqref{Hom2d est remainder} by duality.
\end{proof}
\begin{remark}
    Similarly as in the 3D case, it is crucial that the second relation in \eqref{thm:Hom2D restr gamma,alpha} is fulfilled. 
    In the case $\alpha =2$, we would only obtain $\delta = 0$ in the proof of \Cref{lem:Hom2d mom exts}. Again, this is expected to force the limit equation to be of Brinkman type, however, also this question in 2D is to date completely open.
\end{remark}
\subsection{Homogenization limit}\label{subsec:HomLim2D}
In this subsection we perform the homogenization limit $\eps \to 0$ and conclude the proof of \Cref{thm:Hom} in the case $d=2$.
Similarly as in the 3D case (see \Cref{lem:Hom3d weak lim}), combining \eqref{Hom2d bds exts I}, \eqref{Hom2d cont exts}, \eqref{Hom2d parab exts}, \Cref{lem:Hom2d mom exts}, and the Sobolev embedding $W^{1,2}(\Omega)\hookrightarrow L^{\infty^-}(\Omega)$ yields, up to a subsequence,
\begin{equation}\label{Hom2d weak cv}
        \begin{aligned}
            &\trhoeps \weak \rho \quad \text{in } L^{(2\gamma-1)^-}(\OmegaT),
            \quad
            \trhoeps \to \rho \quad \text{in } \Cw([0,T];L^\gamma(\Omega)),
            \\
            &\tuveceps \weak \uvec \quad \text{in } L^2(0,T;W^{1,2}_0(\Omega)),
            \quad
            \Phi_\eps^{\mathrm{T}}\trhoeps\tuveceps \to \rho \uvec \quad \text{in } \Cw([0,T];L^{\frac{2\gamma}{\gamma-1}}(\Omega)),
            \\
            &\trhoeps \nablax \Eeps \ceps \weakstar \overline{\rho\nablax c} \quad \text{in } L^\infty(0,T;L^{\frac{2\gamma}{\gamma+2}}(\Omega)),
            \quad 
            \Peff(\trhoeps)\weak \overline{\Peff} \quad \text{in } L^{\frac{(2\gamma-1)^-}{\gamma}}(\OmegaT),
            \\
            &\Eeps\ceps \weakstar c \quad \text{in } L^\infty(0,T;W^{1,2}(\Omega)),
            \quad
            \Eeps\ceps \to c \quad \text{in } C([0,T];L^{2}(\Omega)),
            \\
            &\tceps \to c \quad \text{in } C([0,T];L^2(\Omega)),
            \quad
            \partial_t\tceps \weak \partial_t c \quad \text{in } L^2(\OmegaT),
        \end{aligned}
    \end{equation}
    with $(\rho,\uvec,c)$ satisfying
    \begin{align}
        &\int_0^T \int_\Omega
        \rho \partial_t \varphi + \rho\uvec \cdot \nablax\varphi
        \, \dd x \, \dd t
        = 0,\label{Hom2d limit cont}
        \\
        &\int_0^T \int_\Omega 
        \paracoup c \partial_t \varphi - \kappa \nablax c \cdot \nablax \varphi - \coup (c-\rho) \varphi 
        \, \dd x \, \dd t = 0,\label{Hom2d limit parab}
    \end{align}
    for any test function $\varphi \in \mathcal{D}(\RR^2_T)$, and
    \begin{equation}\label{Hom2d limit mom}
        \begin{aligned}
            &\int_0^T \int_{\Omega}
            \rho\uvec \cdot \partial_t \varphivec + \rho\uvec\otimes\uvec :\nablax \varphivec + \overline{p_\coup}\div\varphivec 
            \, \dd x \, \dd t
            \\
            &\quad =
            \int_0^T \int_\Omega 
            \Svisc(\nablax\uvec):\nablax\uvec - \coup \overline{\rho\nablax c} \cdot \varphivec
            \, \dd x \, \dd t
        \end{aligned}
    \end{equation}
    for any test function $\varphivec\in \mathcal{D}(\OmegaT;\RR^2)$.
    Moreover, the triplet $(\rho, \uvec, c)$ satisfies the regularity \eqref{wkSol reg} with $\Omegaeps$ replaced by $\Omega$ as well as the initial conditions
    \begin{equation}\label{Hom2d limit IC}
        \rho(0)=\rho_0,
        \qquad 
        (\rho\uvec)(0)= (\rho\uvec)_0,
        \qquad 
        c(0)=c_0
        \qquad \text{a.e.~in } \Omega.
    \end{equation}
To identify the limit quantities $\overline{\Peff}$ and $\overline{\rho\nablax c}$, we proceed as in the 3D case and show the strong convergence for the density $\trhoeps$ by exploiting the weak compactness property of the effective viscous flux, as well as the technique from \cite{Feireisl2002}.
\begin{lemma}\label{lem:Hom2d str cv density}
    We have
    \begin{equation*}
        \trhoeps \to \rho \quad \text{in } L^{(2\gamma-1)^-}(\OmegaT),
    \end{equation*}
    and, in particular,
    \begin{equation*}
        \overline{\Peff} = \Peff(\rho),
        \quad 
        \overline{\rho\nablax c} = \rho \nablax c \qquad \text{a.e.~in } \OmegaT.
    \end{equation*}
\end{lemma}
\begin{proof}
    Let $T_k$ be defined as in the proof of \Cref{lem:Hom3d str cv density}.
    Then, by similar arguments as in the aforementioned proof we have that
    \begin{equation*}
        \partial_tT_k(\trhoeps) + \div( T_k(\trhoeps)\tuveceps) + \big( (T_k^\prime(\trhoeps)\trhoeps - T_k(\trhoeps)\Big) \div\tuveceps = 0 \qquad \text{in } \mathcal{D}^\prime(\RR^2_T),
    \end{equation*}
    \begin{equation*}
        T_k(\trhoeps) \to \overline{T_k} \quad \text{in } \Cw([0,T];L^{\infty^-}(\Omega)),
    \end{equation*}
    for some $\overline{T_k}\in \Cw([0,T];L^{\infty^-}(\Omega))$, and
    \begin{equation*}
        T_k(\trhoeps) \weakstar \overline{T_k} \quad \text{in } L^\infty(0,T;L^\infty(\Omega)),
        \qquad 
        \big(T_k^\prime(\trhoeps)\trhoeps - T_k(\trhoeps)\big) \div\tuveceps \weak f \quad \text{in } L^2(\OmegaT),
    \end{equation*}
    for some $f \in L^2(\OmegaT)$.
    Thanks to these relations, \Cref{lem:Hom2d mom exts}, \eqref{Hom2d weak cv}, and \eqref{thm:Hom2D restr gamma,alpha}, we find some $\delta\in (0,\tfrac14)$ small enough such that we may apply \cite[Proposition~A.1]{OschmannWendt2026-exist} with
        \begin{equation*}
        q=\frac{1}{\delta},
        \quad 
        r=\frac{2\gamma-1}{\gamma} - \delta,
        \quad 
        s=\frac{2\gamma}{\gamma+2},
        \quad
        z=\frac{2\gamma}{\gamma+1},
        \quad 
        \sigma = 2,
        \quad 
        w = r^\prime.
    \end{equation*}
    We obtain the relation
    \begin{equation*}
        \lim\limits_{\eps\to0}
        \int_0^T \int_\Omega 
        \varphi T_k(\trhoeps) 
        \Big(\Peff(\trhoeps) - \big(\eta + \mu\big)\div\tuveceps\Big) 
        \, \dd x \, \dd t
        =
        \int_0^T \int_\Omega 
        \varphi \overline{T_k}
        \Big( 
        \overline{\Peff}
        -
        \big(
        \eta + \mu
        \big)
        \div\uvec
        \Big)
        \, \dd x \, \dd t
    \end{equation*}
    for any test function $\varphi \in \mathcal{D}(\OmegaT)$.
    With this relation at hand, we can proceed similarly as in the proof of \Cref{lem:Hom3d str cv density} in order to conclude the strong convergence of $\trhoeps$.
\end{proof}
\begin{proof}[Proof of \Cref{thm:Hom} in the case $d=2$]
    In view of \eqref{Hom2d weak cv}--\eqref{Hom2d limit IC} and \Cref{lem:Hom2d str cv density}, we only have to show that the triplet $(\rho,\uvec,c)$ satisfies the energy inequality \eqref{wkSol energy}.
    This follows from \eqref{Hom2d bds exts II}, \eqref{Hom2d weak cv}, and \Cref{lem:Hom2d str cv density} by using the same arguments as in the 3D case (see end of \Cref{sec:Hom3D}).
\end{proof}

\section{Conclusions and future directions}\label{sec:Concl}
In this work, we considered the rNSKE on a bounded domain perforated by small obstacles.
This relaxation system is an approximate model for the isothermal compressible NSKE provided the relaxation coefficients are chosen appropriately.
In particular, it serves as a reasonable model for a compressible viscous fluid that can occur in two different phases provided a suitable non-monotone pressure function of Van-der-Waals type is used.
Focusing on the case of very tiny holes, that is, the size of the obstacles decrease much faster than their mutual distances, we have shown that the effective behavior of the two-phase fluid is not influenced by the obstacles.
Our result applies for spatial dimension two and three, as well as for a two-phase setting by allowing the underlying pressure function to be non-monotone.
Recently, in \cite{HoeferNecasovaOschmann2026}, a homogenization result for the compressible single-phase NSE in the regime of large holes including convergence rates has been obtained by using a relative energy approach.
We plan to generalize this result to the rNSKE in a two-phase setting by using the relative energy inequality for the rNSKE which has been derived in \cite{ChaudhuriRohdeWendt2025}. Moreover, we plan to investigate the combined limit of homogenization and relaxation to make the diagram ``Homogenization for rNSKE - Relaxation limit to NSK'' commutative.

\appendix
\section{The Bogovski\u{i} operator in 2D and its extension to negative Sobolev spaces}\label{App:Bog2D}
In this section we give a proof of \Cref{lem:Bogeps 2D} and extend the Bogovski\u{\i} constructed in \cite[Appendix~A]{NecasovaOschmann2023} to negative Sobolev spaces.
One integral part of this construction is the Bogovski\u{\i} operator on $\Omega$, the existence of which is classical:
\begin{lemma}[See~{\cite[Theorem~11.17]{FeireislNovotny2017singlim}}]\label{lem:Bog(Appendix)}
    Let $\Omega \subseteq \RR^2$ be a bounded Lipschitz domain and let $s \in (1,\infty)$.
    Then there exists a bounded linear operator $\Bog \colon L^s_0(\Omega) \to W^{1,s}_0(\Omega;\RR^2)$ such that for any $f \in L^s_0(\Omega)$,
    \begin{align*}
        &\div\Bog(f) = f,
        && \|\Bog(f)\|_{W^{1,s}(\Omega)}\leq C\|f\|_{L^s(\Omega)}.
    \end{align*}
    Moreover, this operator can be extended to a linear operator
    \begin{align*}
        \Bog\colon \big[\dot{W}^{1,s^\prime}(\Omega)\big]^\ast
        := 
        \{f \in \big[W^{1,s^\prime}(\Omega)\big]^\ast : \langle f,1\rangle = 0\}
        \to
        L^s(\Omega;\RR^2)
    \end{align*}
    such that, for any $f \in \big[\dot{W}^{1,s^\prime}(\Omega)\big]^\ast$,
    \begin{align*}
        -\int_\Omega \Bog(f) \cdot \nablax \phi \, \dd x 
        =
        \langle f,\phi\rangle \quad \forall\, \phi \in W^{1,s^\prime}(\Omega),
        \qquad 
        \|\Bog(f)\|_{L^s(\Omega)} \leq C \|f\|_{[\dot{W}^{1,s^\prime}(\Omega)]^\ast}.
    \end{align*}
\end{lemma}
We also need a suitable family of Bogovski\u{\i} operators $\Bogepsi$ on the annuli corresponding to the holes of the perforated domain $\Omegaeps$.
\begin{lemma}\label{lem:Bogepsi(Appendix)}
    For $\eps \in (0,1)$, let $\Omegaeps$ be defined as in $\eqref{def:domain}$ with $d=2$, $s \in (1,\infty)$, and denote $\varpi_\eps := \eps^{\frac{\alpha}{2}}a_\eps^{-1}$, $x_i^\eps := \eps x_i$ (see \Cref{subsec:perforated domain}).
    Then, for any $i \in K_\eps$ there exists some bounded linear operator
    \begin{align*}
        \Bogepsi\colon L^s_0(\Aepsi)
        \to 
        W^{1,s}_0 (\Aepsi;\RR^2)
    \end{align*}
    with $\Aepsi := B_{\varpi_\eps a_\eps}(x_i^\eps)\setminus B_{a_\eps}(x_i^\eps)$, such that for any $f \in L^s_0(\Aepsi)$,
    \begin{align*}
        &\div\Bogepsi(f) = f
        &&\|\Bogeps(f)\|_{W^{1,s}(\Aepsi)} \leq C\|f\|_{L^s(\Aepsi)}.
    \end{align*}
    Moreover, this operator can be extended to a linear operator
    \begin{align*}
        \Bog\colon \big[\dot{W}^{1,s^\prime}(\Aepsi)\big]^\ast
        := 
        \{f \in \big[W^{1,s^\prime}(\Aepsi)\big]^\ast : \langle f,1\rangle = 0\}
        \to
        L^s(\Aepsi;\RR^2),
    \end{align*}
    such that, for any $f \in \big[\dot{W}^{1,s^\prime}(\Aepsi)\big]^\ast$,
    \begin{align*}
        -\int_{\Aepsi} \Bog(f) \cdot \nablax \phi \, \dd x 
        =
        \langle f,\phi\rangle \quad \forall\, \phi \in W^{1,s^\prime}(\Aepsi),
        \qquad 
        \|\Bog(f)\|_{L^s(\Aepsi)} \leq C \|f\|_{[\dot{W}^{1,s^\prime}(\Aepsi)]^\ast}.
    \end{align*}
\end{lemma}
\begin{proof}
    First, we notice that the proof of \cite[Theorem~1.1]{LuSchwarzacher2018} also works in spatial dimension two.
    Moreover, we have for any $i \in K_\eps$ that $\Aepsi$ is a $c$-John domain (see~\cite{DieningRuzickaSchumacher2010} for a definition of $c$-John domains) with a constant $c>0$ that does not depend on $\eps >0$.
    The existence of the operators $\Bogepsi$ thus follows from an application of the 2D version of \cite[Theorem~1.1]{LuSchwarzacher2018}.
\end{proof}
We emphasize that the constant $C>0$ in \Cref{lem:Bogepsi(Appendix)} does not depend on $\eps$.
With the Bogovski\u{\i} operators $\Bog$ and $\Bogepsi$ at hand, let us now prove \Cref{lem:Bogeps 2D}.
\begin{proof}[Proof of \Cref{lem:Bogeps 2D}]
    The proof follows essentially the lines of \cite[Proposition~2.2]{LuSchwarzacher2018}, but in contrast to there, we use different cut-off functions like done in \cite{Bravin2024, NecasovaOschmann2023}. Therefore, we give a full proof here.\\
    
    Let the notations of \Cref{lem:Bogepsi(Appendix)} hold true and let us fix $s\in(1,\infty)$.
    We introduce cut-off functions 
    \begin{align*}
        &\Xieps(r)
        :=
        \begin{cases}
            1 & 0 \leq r <a_\eps,
            \\
            \frac{\log(\varpi_\eps a_\eps)-\log(r)}{\log(\varpi_\eps a_\eps) - \log(a_\eps)} &a_\eps\leq r <\varpi_\eps a_\eps,
            \\
            0 &\varpi_\eps a_\eps \leq r <\infty,
        \end{cases}
        &&\thetaeps(r)
        :=
        \begin{cases}
            1 & 0\leq r <\frac{\varpi_\eps a_\eps}{2},
            \\
            \frac{2}{\varpi_\eps a_\eps}(\varpi_\eps a_\eps - r) &\frac{\varpi_\eps a_\eps}{2}\leq r < \varpi_\eps a_\eps,
            \\
            0 & \varpi_\eps a_\eps \leq r < \infty.
        \end{cases}
    \end{align*}
    By an elementary calculation we verify
    \begin{align}\label{bds Xieps}
        &\|\Xieps\|_{L^\infty(\RR^2)}\leq C,
        &&\|\nablax\Xieps\|_{L^s(\RR^2)}^s \leq C 
        \begin{cases}
            \frac{a_\eps^{2-s}}{|\log \varpi_\eps|^s}|\varpi_\eps^{2-s}-1| &s \neq 2,
            \\
            |\log \varpi_\eps|^{-1} &s=2,
        \end{cases}
    \end{align}
    and
    \begin{align}\label{bds thetaeps}
        &\|\thetaeps\|_{L^\infty(\RR^2)} \leq C,
        &&\|\nablax \thetaeps\|_{L^s(\RR^2)}^s \leq C(\varpi_\eps a_\eps )^{2-s},
        &&&\|\nablax \thetaeps\|_{L^\infty(\RR^2)}\leq C(\varpi_\eps a_\eps)^{-1}.
    \end{align}
    For $i \in K_\eps$, we introduce the corresponding single-hole cut-off functions via
    \begin{align*}
        &\Xiepsi(x):= \Xieps(|x-\xepsi|),
        &&\thetaepsi(x) := \thetaeps(|x-\xepsi|)
        &&\text{for~}x \in \RR^2.
    \end{align*}
    With these cut-off functions, we then define 
    \begin{align*}
        \Lepsi \vvec 
        :=
        \thetaepsi \Big( \vvec - \fint_{\Bepsi} \vvec\Big) + \Xiepsi \fint_{\Bepsi}\vvec 
        \qquad \forall\, \vvec \in W^{1,s}(\Bepsi),
    \end{align*}
    where $\Bepsi :=B_{\varpi_\eps a_\eps}(\xepsi)$.
    In \cite[Appendix~A]{NecasovaOschmann2023} it has been shown that this definition gives rise to a bounded linear operator $\Lepsi \colon W^{1,s}(\Bepsi) \to W^{1,s}_0(\Bepsi)$ satisfying
    \begin{align}\label{avg div Lepsi}
        \int_{\Aepsi} \div(\Lepsi \vvec)\, \dd x = 0 
        \qquad \forall\, \vvec \in W^{1,s}(\Bepsi).
    \end{align}
    For $f \in L^s_0(\Omegaeps)$ we then define
    \begin{align*}
        \Bogeps(f):= \Bog(\tilde{f}) - \sum\limits_{i\in K_\eps} \Lepsi \Bog(\tilde{f}) - \sum\limits_{i \in K_\eps} \Bogepsi\big(\div\Lepsi\Bog(\tilde{f})\big).
    \end{align*}
    Note that the third term on the right-hand side in the preceding definition is well-defined due to \eqref{avg div Lepsi}.
    In \cite[Appendix~A]{NecasovaOschmann2023} it has been shown that this definition gives rise to a bounded linear operator $\Bogeps \colon L^s_0(\Omegaeps) \to W^{1,s}_0(\Omegaeps;\RR^2)$ satisfying \eqref{lem:Bogeps 2D non-ext}.
    Now we wish to extend this operator to a linear operator as in \eqref{lem:Bogeps 2D ext}.
    To do so, we fix $\mathbf{g}\in L^s(\Omegaeps;\RR^2)$ and assume that $\div \mathbf{g}\in L^s(\Omegaeps)$, $\mathbf{g}\cdot \mathbf{n}_{\partial\Omegaeps} = 0$ in the sense of traces on $\partial\Omegaeps$.
    Our goal is to show the estimate
    \begin{align}\label{target est}
        \|\Bogeps (\div\mathbf{g})\|_{L^s(\Omegaeps)} 
        \leq 
        C\|\mathbf{g}\|_{L^s(\Omegaeps)}
    \end{align}
    for some constant $C>0$ depending neither on $\eps$ nor on $\|\div\mathbf{g}\|_{L^s(\Omegaeps)}$.
    To verify this estimate, we first notice that
    \begin{equation}\label{to be estimated for target est}
    \begin{aligned}
        \|\Bogeps(\div\gvec)\|_{L^s(\Omegaeps)}
        &\leq 
        \|\Bog(\div\tilde{\gvec})\|_{L^s(\Omega)} 
        +
        \sum\limits_{i\in K_\eps} \|\Lepsi \Bog(\div\tilde{\gvec})\|_{L^s(\Bepsi)}
        \\
        &\qquad
        + \sum\limits_{i\in K_\eps} \|\Bogepsi \big(\div \Lepsi \Bog(\div\tilde{\gvec})\big) \|_{L^s(\Aepsi)},
    \end{aligned}
    \end{equation}
    where we have used $\widetilde{\div\gvec} = \div \tilde{\gvec}$.
    We estimate the right-hand side of \eqref{to be estimated for target est} term by term.
    For the first term, we use \Cref{lem:Bog(Appendix)} and the fact that $\tilde{\gvec}\cdot \mathbf{n}_{\partial\Omega}=0$ in the sense of traces to estimate
    \begin{align}\label{LHS I}
        \|\Bog(\div\tilde{\gvec})\|_{L^s(\Omega)}
        \leq 
        C\|\gvec\|_{L^s(\Omegaeps).}
    \end{align}
    For the second and third term on the left-hand side in \eqref{to be estimated for target est}, we notice by Hölder's inequality
    \begin{align}\label{Est on avg}
        \Big|\fint_{\Bepsi} \Bog(\div \tilde{\gvec})\Big|
        \leq 
        |\Bepsi|^{-\frac{1}{s}} \|\Bog(\div  \tilde{\gvec})\|_{L^s(\Bepsi)}
        \leq 
        C (\varpi_\eps a_\eps)^{-\frac{2}{s}}\|\Bog(\div\tilde{\gvec})\|_{L^s(\Bepsi)}
    \end{align}
    and by Poincaré's inequality
    \begin{align}\label{PC}
        \|\varphi\|_{L^s(\Aepsi)} \leq C \varpi_\eps a_\eps \|\nablax\varphi\|_{L^{s}(\Aepsi)} \qquad \forall\, \varphi \in W^{1,s}_0(\Aepsi).
    \end{align}
    With \eqref{bds Xieps}, \eqref{bds thetaeps}, and \eqref{Est on avg} we estimate for the second term on the left-hand side in \eqref{to be estimated for target est}
    \begin{align*}
        \|\Lepsi\Bog(\div\tilde{\gvec})\|_{L^s(\Bepsi)}
        &\leq 
        \big(\|\thetaepsi\|_{L^\infty(\RR^2)} + \|\Xiepsi\|_{L^\infty(\RR^2)}\big) \Big(\|\Bog(\div\tilde{\gvec})\|_{L^s(\Bepsi)} + \big\|\fint_{\Bepsi} \Bog(\div\tilde{\gvec})\big\|_{L^s(\Bepsi)} \Big)
        \\
        &\leq 
        C\|\Bog(\div\tilde{\gvec})\|_{L^s(\Bepsi)}.
    \end{align*}
    Summing over $i\in K_\eps$ yields
    \begin{align}\label{LHS II}
        \sum\limits_{i\in K_\eps} \|\Lepsi \Bog(\div\tilde{\gvec})\|_{L^s(\Bepsi)}
        \leq 
        C\sum\limits_{i\in K_\eps} \|\Bog(\div\tilde{\gvec})\|_{L^s(\Bepsi)}
        \leq 
        C\|\Bog(\div\tilde{\gvec})\|_{L^s(\Omega)}
        \leq 
        C\|\gvec\|_{L^s(\Omegaeps)},
    \end{align}
    where we have used \Cref{lem:Bog(Appendix)} to obtain the last inequality.
    For the third term on the left-hand side of \eqref{to be estimated for target est} we calculate
    \begin{align*}
        \div\Lepsi\Bog(\div\tilde{\gvec})
        %\\
        %&=
        %\nablax \thetaepsi \cdot\Big(\Bog(\div\tilde{\gvec}) - \fint_{\Bepsi} \Bog(\div\tilde{\gvec})\Big)
        %+ \thetaepsi \div\Bog(\div\tilde{\gvec})
        %+ \nablax\Xiepsi \cdot \fint_{\Bepsi} \Bog(\div\tilde{\gvec})
        %\\
        =
        \div(\thetaepsi \tilde{\gvec}) + \div\big( \thetaepsi [ \Bog(\div\tilde{\gvec}) - \tilde{\gvec} ]\big) - \div \left( (\thetaepsi - \Xiepsi) \fint_{\Bepsi} \Bog(\div\tilde{\gvec}) \right).
    \end{align*}
    Thus, we have 
    \begin{align*}
        &\big\|\Bogepsi\big( \div\Lepsi\Bog(\div\tilde{\gvec})\big)\big\|_{L^s(\Aepsi)}
        \\
        &\leq 
        \big\|\Bogepsi\big(\div(\thetaepsi\tilde{\gvec})\big)\big\|_{L^s(\Aepsi)}
        +
        \big\|\Bogepsi\big(\div [ \thetaepsi(\Bog(\div\tilde{\gvec}) - \tilde{\gvec}) ] \big)\big\|_{L^s(\Aepsi)}
        \\
        &\qquad +
        \big\|\Bogepsi\big(\div [ (\thetaepsi - \Xiepsi) \fint_{\Bepsi} \Bog(\div\tilde{\gvec}) ] \big)\|_{L^s(\Aepsi)}
        %=: I_1^\eps+I_2^\eps+I_3^\eps.
    \end{align*}
    By \Cref{lem:Bogepsi(Appendix)} and \eqref{bds thetaeps} we have
    \begin{align*}
        \big\|\Bogepsi\big(\div(\thetaepsi\tilde{\gvec})\big)\big\|_{L^s(\Aepsi)}
        \leq 
        C \|\thetaepsi\tilde{\gvec}\|_{L^s(\Aepsi)}
        \leq 
        C\|\tilde{\gvec}\|_{L^s(\Aepsi)}.
    \end{align*}
    By \eqref{bds thetaeps} and \eqref{PC} we estimate
    \begin{align*}
        &\big\|\Bogepsi\big(\div [ \thetaepsi(\Bog(\div\tilde{\gvec}) - \tilde{\gvec}) ] \big)\big\|_{L^s(\Aepsi)}
        =
        \big\|\Bogepsi\big(\nablax \thetaepsi \cdot [ \Bog(\div\tilde{\gvec}) - \tilde{\gvec} ] \big)\big)\big\|_{L^s(\Aepsi)}
        \\
        &\leq 
        C \varpi_\eps a_\eps \big\|\nablax \Bogepsi\big(\nablax \thetaepsi \cdot [ \Bog(\div\tilde{\gvec}) - \tilde{\gvec} ] \big)\big)\big\|_{L^s(\Aepsi)}
        \\
        &\leq 
        C\varpi_\eps a_\eps \|\nablax\thetaepsi\cdot [ \Bog(\div\tilde{\gvec})-\tilde{\gvec} ] \|_{L^s(\Aepsi)}
        \\
        &\leq 
        C\varpi_\eps a_\eps \|\nablax \thetaepsi\|_{L^\infty(\RR^2)} \big( \|\Bog(\div\tilde{\gvec})\|_{L^s(\Aepsi)} + \|\tilde{\gvec}\|_{L^s(\Aepsi)}\big)
        \\
        &\leq 
        C\big( \|\Bog(\div\tilde{\gvec})\|_{L^s(\Aepsi)} + \|\tilde{\gvec}\|_{L^s(\Aepsi)} \big),
    \end{align*}
    where we have used $\div\Bogepsi(\div\tilde{\gvec})=\tilde{\gvec}$ to obtain the first equality.
    Due to \eqref{bds thetaeps}, \eqref{Est on avg}, and \eqref{PC} we have
    \begin{align*}
        &\big\|\Bogepsi\big(\div [ (\thetaepsi - \Xiepsi) \fint_{\Bepsi} \Bog(\div\tilde{\gvec}) ] \big)\|_{L^s(\Aepsi)}
        =
        \big\|\Bogepsi\big(  \nablax(\thetaepsi - \Xiepsi) \cdot \fint_{\Bepsi} \Bog(\div\tilde{\gvec}) \big)\|_{L^s(\Aepsi)}
        \\
        &\leq 
        C \varpi_\eps a_\eps \big\|\nablax \Bogepsi\big(  \nablax(\thetaepsi - \Xiepsi) \cdot \fint_{\Bepsi} \Bog(\div\tilde{\gvec}) \big)\|_{L^s(\Aepsi)}
        \\
        &\leq 
        C\varpi_\eps a_\eps \|\nablax (\thetaepsi - \Xiepsi)\|_{L^s(\Aepsi)} \Big|\fint_{\Bepsi}\Bog(\div\tilde{\gvec})\Big|
        \\
        &\leq
        C (\varpi_\eps a_\eps)^{-\frac{2-s}{s}}\big( \|\nablax \thetaepsi\|_{L^s(\RR^2)} + \|\nablax \Xiepsi\|_{L^s(\RR^2)}\big) \|\Bog(\div\tilde{\gvec})\|_{L^s(\Bepsi)}
        \\
        &\leq 
        C\big( 1 + (\varpi_\eps a_\eps)^{-\frac{2-s}{s}}\|\nablax \Xiepsi\|_{L^s(\RR^2)}\big) \|\Bog(\div\tilde{\gvec})\|_{L^s(\Bepsi)}.
    \end{align*}
    Thanks to \eqref{bds Xieps}, we have for any $s \in (1,2)$ that
    \begin{align*}
        (\varpi_\eps a_\eps)^{-\frac{2-s}{s}}\|\nablax \Xiepsi\|_{L^s(\RR^2)}
        &\leq 
        C(\varpi_\eps a_\eps)^{-\frac{2-s}{s}} a_\eps^{\frac{2-s}{s}} \frac{1}{|\log \varpi_\eps|} \big| \varpi_\eps^{2-s} - 1 \big|^{\frac{1}{s}}
        \\
        &\leq 
        C \frac{1}{|\log \varpi_\eps|} \big| 1 - \varpi_\eps^{s-2} \big|^{\frac{1}{s}}
        \leq 
        C
    \end{align*}
    since $\varpi_\eps \xrightarrow{\eps \to 0} \infty$.
    For $s=2$, we evidently have by \eqref{bds Xieps}
    \begin{align*}
        \|\nablax \Xiepsi\|_{L^2(\RR^2)}
        \leq 
        C \frac{1}{|\log \varpi_\eps|^\frac12} \leq C
    \end{align*}
    since $\varpi_\eps \xrightarrow{\eps \to 0}\infty$.
    In total, we have shown for any $s \in (1,2]$ that
    \begin{align*}
        \|\Bogepsi \big(\div \Lepsi \Bog(\div\tilde{\gvec})\big) \|_{L^s(\Aepsi)}
        \leq 
        C \big( \|\Bog(\div\tilde{\gvec})\|_{L^s(\Aepsi)} + \|\tilde{\gvec}\|_{L^s(\Aepsi)}\big)
    \end{align*}
    and summing over $i\in K_\eps$ yields
    \begin{equation}\label{LHS III}
    \begin{aligned}
        &\sum\limits_{i\in K_\eps}\|\Bogepsi \big(\div \Lepsi \Bog(\div\tilde{\gvec})\big) \|_{L^s(\Aepsi)}
        \leq 
        C \sum\limits_{i\in K_\eps} \big( \|\Bog(\div\tilde{\gvec})\|_{L^s(\Aepsi)} + \|\tilde{\gvec}\|_{L^s(\Aepsi)}\big)
        \\
        &\leq C \big( \|\Bog(\div\tilde{\gvec})\|_{L^s(\Omega)} + \|\gvec\|_{L^s(\Omegaeps)} \big)
        \leq 
        C\|\gvec\|_{L^s(\Omegaeps)},
    \end{aligned}
    \end{equation}
    where we have used \Cref{lem:Bog(Appendix)} to deduce the last inequality.
    Combining \eqref{LHS I}, \eqref{LHS II}, and \eqref{LHS III} yields precisely \eqref{target est}.
    Having shown \eqref{target est} for $s \in (1,2]$, the extension in \eqref{lem:Bogeps 2D ext} follows by a density argument similarly as in \cite[Section~11.6]{FeireislNovotny2017singlim}. 
\end{proof}

\section*{Acknowledgments}
{\it F. O. has been supported by the Primus grant PRIMUS 26/SCI/026. F. W. acknowledges funding by Deutsche Forschungsgemeinschaft (DFG, German Research Foundation) under Germany's Excellence Strategy - EXC 2075 - 390740016.}

\bibliographystyle{plain}
\bibliography{Lit-Master}
\end{document}